\documentclass[review, compress, 3p]{elsarticle}
\usepackage{lineno,hyperref}
\usepackage{amssymb,amsfonts}
\usepackage{amsmath,amsthm}
\usepackage{algorithm,algorithmic}
\usepackage{caption}
\usepackage{graphicx,subfigure}
\usepackage{float,rotating}
\usepackage{booktabs,array}
\usepackage{multirow,multicol}
\usepackage{lscape}
\usepackage{epstopdf}
\usepackage{color}
\usepackage{setspace}
\usepackage{natbib}
\usepackage{bm}
\usepackage{url,doi}
\allowdisplaybreaks

\makeatletter
\@addtoreset{equation}{section}
\makeatother

\newtheorem{theorem}{Theorem}[section]
\newtheorem{lemma}{Lemma}[section]

\newtheorem{remark}{Remark}

\newtheorem{corollary}{Corollary}[section]
\newtheorem{assumption}{Assumption}

\begin{document}

\begin{frontmatter}

\title{A Generalized Block Circulant Preconditioner for Crank-Nicolson All-at-Once Systems with Applications to Option Pricing PDEs}

\author[addr1]{Yong-Liang Zhao\fnref{equal}}
\ead{ylzhaofde@sina.com}

\author[addr2]{Yao Li\fnref{equal}}
\ead{1623646703@qq.com}

\author[addr2]{Xian-Ming Gu\corref{correspondingauthor}}
\cortext[correspondingauthor]{Corresponding author}
\ead{guxianming@live.cn}

\author[addr3]{Cornelis W. Oosterlee}
\ead{c.w.oosterlee@uu.nl}
\fntext[equal]{These authors contributed equally to this work.}

\address[addr1]{School of Mathematical Sciences,
Sichuan Normal University, Chengdu, Sichuan 611731, P.R. China}
\address[addr2]{School of Mathematics, Southwestern University of Finance and Economics, Chengdu, Sichuan 611130, P.R. China}
\address[addr3]{Mathematical Institute, Utrecht University, Budapestlaan 6, 3584 CD Utrecht, The Netherlands}

\begin{abstract}
The Crank--Nicolson (CN) method is a widely used time integration scheme for evolutionary partial differential equations (PDEs) arising in various scientific and engineering disciplines. Since the numerical solution at each time level depends on the solution at the previous time level, the resulting discretization is inherently sequential and therefore difficult to parallelize in time. In this paper, we develop an all-at-once formulation of the CN discretization together with a generalized block circulant preconditioner that enables an efficient parallel-in-time solution within a Krylov subspace framework. We establish a detailed spectral analysis of the preconditioned system, proving that most eigenvalues are equal to $1$, while the remaining eigenvalues are confined to the annulus:
\begin{equation*}
\left\{
z\in\mathbb{C}:
\frac{1}{1+\alpha}<|z|<\frac{1}{1-\alpha},
\ \Re(z)>0
\right\},
\end{equation*}
where $0<\alpha<1$ is a free parameter. 
Besides, the efficient implementation of the proposed preconditioner is described. Given certain conditions, we prove that the preconditioned GMRES($m$) method achieves a fast convergence rate independent of discretization stepsizes from the residual point of view. Finally, we verify both theoretical findings and the efficacy of the proposed preconditioner via numerical experiments on financial option pricing PDEs (even with variable coefficients).
\end{abstract}


\end{frontmatter}


\section{Introduction}
\label{sec1}

Evolutionary partial differential equations (PDEs) arise in many scientific and engineering applications, including fluid dynamics, image processing, physics-based animation, mechanical systems, earth sciences, and mathematical finance \cite{Ascher08,Duffy22}. In this work, we consider the numerical solution of the class of evolutionary PDEs
\begin{equation}
\begin{cases}
\partial_{t}u(\bm{x},t)=\mathcal{L}u(\bm{x},t)+f(\bm{x},t),
&(\bm{x},t)\in\Omega\times(0,T],\\
u(\cdot,0)=u_0,
&\mathrm{in}\ \Omega,
\end{cases}
\label{eq1.1}
\end{equation}
where ${\bm x}=(x_1,\ldots,x_d)$, $\Omega$ is a bounded convex domain in $\mathbb{R}^d$ ($d=1,2,3$) with smooth boundary, and $\mathcal{L}$ denotes a spatial differential operator. Typical examples include
\begin{equation}
\mathcal{L}=
\begin{cases}
\displaystyle
\sum_{j,\ell=1}^{d}
\left[
\frac{\partial}{\partial x_\ell}
\left(
a_{j\ell}(\bm{x})
\frac{\partial}{\partial x_j}
\right)
-
b_\ell(\bm{x})
\frac{\partial}{\partial x_\ell}
\right]
-
c_0(\bm{x}),
&
\text{self-adjoint elliptic operator},
\\[1.5ex]
\displaystyle
\sum_{j=1}^{d}
\kappa_j
\frac{\partial^{\beta_j}}
{\partial |x_j|^{\beta_j}},
&
\text{Riesz fractional operator},
~
\beta_j\in(1,2),
\end{cases}
\label{eq1.2xz}
\vspace{-0.5em}
\end{equation}
whose precise definitions can be found in \cite{Evans2010,Lischke20}. Throughout the paper we assume that problem \eqref{eq1.1} is well-posed and admits a unique solution.

Analytical solutions of evolutionary PDEs are rarely available in practical applications, making efficient numerical methods indispensable. Standard discretization techniques include finite difference, finite element and spectral methods \cite{Tadmor12}. After spatial discretization, the resulting system of ordinary differential equations is commonly advanced in time by a sequential time-stepping method. 
These methods are inherently sequential due to their local historical dependencies.
For 
fine temporal meshes, this sequential character may become a computational bottleneck.

This observation has motivated the development of parallel-in-time (PinT) algorithms, including the inverse Laplace transform method \cite{Sheen03}, the parareal method \cite{Lions2001}, and MGRIT \cite{Falgout14}. Among the various PinT strategies, all-at-once (AaO) methods discretize the entire space--time problem simultaneously, resulting in a large linear system that can be solved efficiently by preconditioned Krylov subspace methods; see \cite{McDonald18,WATHEN2022,gander21paradiag}.

The efficiency of an AaO approach depends strongly on both the underlying time integrator and the corresponding preconditioner. McDonald \emph{et al.}~\cite{McDonald18} introduced a block circulant (BC) preconditioner for AaO systems arising from backward differentiation formulae (BDF) for evolutionary PDEs, while Goddard and Wathen \cite{Goddard19} demonstrated its parallel performance. Lin and Ng \cite{lin2021all} subsequently proposed a generalized block circulant (gBC) preconditioner by introducing a parameter $\alpha\in(0,1)$ into the top-right block of the BC preconditioner. This additional degree of freedom considerably improves the performance when the diffusion coefficient is small. Both approaches rely on a modified diagonalization technique originally proposed by Maday and R{\o}nquist \cite{Maday2008}.

Although several related PinT preconditioners have been developed \cite{Hon2023,Hon2022}, most existing analyses are restricted to time discretizations at integer time levels. Extending these results to schemes involving non-integer time levels is far from straightforward. Moreover, for several important classes of PDEs the spatial discretization matrix possesses zero eigenvalues, rendering the classical BC preconditioner singular. The gBC preconditioner overcomes this difficulty, but the available analysis is limited to the implicit Euler method and symmetric positive definite spatial discretization matrices.

In this paper we consider the Crank--Nicolson (C-N) time discretization \cite{Crank47}, one of the most widely used second-order implicit schemes for evolutionary PDEs. After spatial discretization, the C-N scheme is reformulated as an AaO linear system with a block lower triangular Toeplitz structure. We then adapt the generalized BC preconditioner to this setting and establish a detailed theoretical analysis.

More specifically, we derive sharp bounds for the condition number of the AaO system and investigate the invertibility of the proposed preconditioner. Assuming that the spatial discretization matrix is negative semidefinite, we prove that most eigenvalues of the preconditioned matrix are equal to $1$, while the remaining eigenvalues are confined to a small explicitly characterized region (annulus) in the complex plane. This spectral result accounts for the fast convergence of the preconditioned GMRES($m$) method (where $m$ is called the restart parameter) and is sharper than existing analyses for related PinT preconditioners, such as that in \cite{Gan2022}, where the spatial matrix is assumed to be symmetric positive definite and the eigenvalue moduli are merely bounded in $[1-\epsilon, 1+\epsilon]$ without explicit dependence on $\alpha$ and spatial discretizations. We further derive a mesh-independent convergence estimate for GMRES($m$) under a mild condition on the parameter $\alpha$. In addition, we show how the proposed preconditioner can be implemented efficiently in parallel, discuss its computational complexity and memory requirements, and demonstrate that a slight modification enables its application to financial PDEs with time-dependent coefficients.

The remainder of this paper is organized as follows. Section~\ref{sec2} derives the AaO system associated with the Crank--Nicolson discretization and studies its invertibility and condition number. Section~\ref{sec3} introduces the generalized block circulant preconditioner and presents its spectral analysis together with an efficient parallel implementation. A mesh-independent convergence analysis of preconditioned GMRES($m$) is given in Section~\ref{sec4}. Numerical experiments, including applications to European option pricing, are presented in Section~\ref{sec5}. Finally, conclusions are drawn in Section~\ref{sec6}.
\vspace{-2mm}
\section{The C-N scheme and its AaO system}
\label{sec2}
In this section, we first derive a C-N scheme for approximating Eq.~\eqref{eq1.1} and give some related analyses.
\subsection{The C-N scheme}
For a given positive integer $N_t$, let $\tau = \frac{T}{N_t}$
be the time step size.
Then, we obtain time steps $t_k = k \tau$ ($k = 0,1,\ldots,N_t$).
Let $A \in \mathbb{R}^{N_x \times N_x} \approx \mathcal{L}$ be the spatial discretization matrix
via a finite difference method.

Combining the C-N time integrator, the matrix form of the C-N scheme
of Eq.~\eqref{eq1.1} is derived as
\begin{equation}\label{eq2.1}
\frac{\bm{u}^{k + 1} - \bm{u}^k}{\tau} = A \frac{\bm{u}^{k + 1} + \bm{u}^k}{2} + \bm{f}^{k + \frac{1}{2}},
\quad k = 0,1,\ldots,N_t - 1,
\vspace{-0.5em}
\end{equation}
where $\bm{u}^k$ is a lexicographically ordered vector collecting the approximate
solution of $u(\cdot, t_k)$ over all the space grids.
Similarly, $\bm{f}^{k + \frac{1}{2}}$ is a lexicographically ordered vector collecting $f(\cdot, t_{k + \frac{1}{2}})$ over all the space grids.
Here $t_{k + \frac{1}{2}} = \left( k + \frac{1}{2} \right) \tau$. 
To investigate the unconditional stability and convergence of the C-N scheme (\ref{eq2.1}),
we need the following lemma.
\begin{lemma}{\rm(\cite[pp. 168-169]{Hairer96})}
Let the rational function $R(z)$ be bounded for $\Re(z)\leq 0$ and assume that
the matrix $A\in\mathbb{R}^{N_x\times N_x}$ is negative semi-definite, i.e., ${\bm y}^{\top}A{\bm y}\leq 0$ for all ${\bm y}\in
\mathbb{R}^{N_x}$. Then, we have
\begin{equation}
\|R(A)\|_2\leq \sup_{\Re(z)\leq 0}|R(z)|,
\end{equation}
where $\Re(\cdot)$ is the real part of a complex number.
\label{lem2.1}
\end{lemma}
At this stage, we can theoretically prove that the C-N scheme (\ref{eq2.1}) is unconditionally stable.
\begin{theorem}
If the spatial matrix $A$ is negative semi-definite, then the C-N scheme (\ref{eq2.1}) is unconditionally stable.
\label{the2.1}
\end{theorem}
\begin{proof} Let ${\bm u}^k$ and $\widetilde{{\bm u}}^k$ be the exact and numerical
solutions of (\ref{eq2.1}), and the error ${\bm e}^k = {\bm u}^k - \widetilde{{\bm u}}^k$. Then,
from (\ref{eq2.1}), we get error equation:
\begin{equation}
\left(I_s - \frac{\tau}{2}A\right){\bm e}^{k+1} = \left(I_s + \frac{\tau}{2}A\right){\bm e}^{k},
\end{equation}
where $I_s$ represents an identity matrix with size $N_x$.

Since $A$ is negative semi-definite, then $I_s - \frac{\tau}{2}A$ is invertible. Therefore, one gets
\begin{equation}
{\bm e}^{k+1} = R\left(\tau A\right){\bm e}^{k}
\end{equation}
with $R(z) := \frac{1+z/2}{1-z/2}$ and $\tau A$ also being negative semi-definite. By Lemma \ref{lem2.1}, it follows that
\begin{equation}
\left\|R\left(\tau A\right)\right\|_2\leq \sup_{\Re(z)\leq 0}|R(z)| = \sup_{\Re(z)= 0}|R(z)| = 1.
\label{eq2.5}
\end{equation}
From (\ref{eq2.5}), it concludes that
\begin{equation}
\|{\bm e}^{k+1}\|_2 = \|R\left(\tau A\right){\bm e}^{k}\|_2\leq
\left\|R\left(\tau A\right)\right\|_2\cdot\|{\bm e}^{k}\|_2\leq \|{\bm e}^{k}\|_2.
\end{equation}
This completes the proof.
\end{proof}

With the help of the same conditions of Theorem \ref{the2.1} and spatial discretization, it is easy to combine
Theorem \ref{the2.1} with Lemma \ref{lem2.1} to present and prove
the convergence of the C-N scheme (\ref{eq2.1}). However, in this section, we do not specify which
spatial discretization is chosen to obtain the spatial matrix, thus here we omit the specific convergence
result of the C-N scheme (\ref{eq2.1}) for clarity.

\subsection{C-N AaO system}
In order to construct the parallel method based on the C-N scheme, we will derive the
so-called AaO linear system. Before deriving the C-N AaO system,
an auxiliary symbols is introduced:
$\bm{0}$ is the zero vector or matrix with suitable size.
Denote
\begin{equation*}
\bm{u} = \left[ \left( \bm{u}^1 \right)^{\top},  \left( \bm{u}^2 \right)^{\top},\ldots, \left( \bm{u}^{N_t} \right)^{\top} \right]^{\top} \quad\mathrm{and}\quad
\bm{f} = \left[ \left( \bm{f}^{\frac{1}{2}} \right)^{\top}, \left( \bm{f}^{\frac{3}{2}} \right)^{\top},
\ldots, \left( \bm{f}^{N_t - \frac{1}{2}} \right)^{\top} \right]^{\top}.
\end{equation*}

Then, the all-at-once system can be written as:
\begin{equation}
\mathcal{M} \bm{u} = \bm{\xi} + \tau\bm{f}:={\bm b},
\label{eq2.2}
\end{equation}
where
\begin{equation*}
\bm{\xi} = \left[ \left[\left( I_s + \frac{1}{2} \tilde{A} \right) \bm{u}^0 \right]^\top, \bm{0}^\top, \ldots,\bm{0}^\top \right]^\top \quad \mathrm{and} \quad
\mathcal{M} = B_1 \otimes I_s - B_2 \otimes \tilde{A}
\end{equation*}
with
$\tilde{A} = \tau A$,  
$B_1 = \mathrm{tridiag}(-1,1,0) \in \mathbb{R}^{N_t \times N_t}$ and
$B_2 = \frac{1}{2}\, \mathrm{tridiag}(1,1,0) \in \mathbb{R}^{N_t \times N_t}$.

If a general sparse direct solver is applied to the system \eqref{eq2.2},
the computational cost is $\mathcal{O}(N_t^3 N^{3}_x)$.
Considering the special structure of $\mathcal{M}$,
if using the Gaussian elimination based block forward substitution method \cite{golub2013matrix} to solve \eqref{eq2.2},
the complexity is of order $\mathcal{O}(N_t N_x^3 + N_t N_x^2)$.
Let
$ \bm{U} = \left[ \bm{u}^1 , \bm{u}^2, \ldots, \bm{u}^{N_t} \right] \in \mathbb{R}^{N_x \times N_t}, $
then the linear system \eqref{eq2.2} can be reformulated as 
\begin{equation*}
    \bm{U} B_1^{\top} - \tilde{A} \bm{U} B_2^{\top}
    = \left[ \left( I_s + \frac{1}{2} \tilde{A} \right) \bm{u}^0 + \tau \bm{f}^{\frac{1}{2}} , \tau \bm{f}^{\frac{3}{2}}, \ldots, 
    \tau \bm{f}^{N_t - \frac{1}{2}} \right].
\end{equation*}
If using the Bartels-Stewart method \cite{Gardiner1992} to solve the generalized Sylvester equation, the computational cost is $\mathcal{O}(N_t^3 + N_x^3)$.
This is still expensive, especially for high-dimensional problems with fine mesh. On the other hand,
we can estimate the condition number of the coefficient matrix of Eq.~\eqref{eq2.2}, this will guide
us to discuss the design of an efficient solver for such a large AaO system.
\begin{theorem}
Suppose that $\Re(\lambda(A)) \leq 0$. Then, the following bounds hold for the critical singular values
of $\mathcal{M}$ in \eqref{eq2.2}:
\begin{equation}
\sigma_{\max}(\mathcal{M})\leq 2 + \tau\|A\|_2,\quad \sigma_{\min}(\mathcal{M})\geq \frac{1}{N_t},
\end{equation}
so that $\mathrm{cond}(\mathcal{M})\leq 2N_t + T\|A\|_2$.
\end{theorem}
\noindent\textbf{Proof}. Since both $\Re(\lambda(A)) \leq 0$ and $\lambda(B_2) > 0$, we may claim that
$\|\mathcal{M}{\bm z}\|_2 \geq \|(B_1\otimes I){\bm z}\|_2$ for any vector ${\bm z}$. In particular,
using also the properties of the Kronecker product, we may estimate $\sigma_{\min}(\mathcal{M})
\geq \sigma_{\min}(B_1)$. The rest of this proof is very similar to the idea in \cite{Dolgov14phd}, we omit the details. \hfill$\Box$

It is natural that the condition number depends linearly on all main properties of the system: number of time steps, time interval length, and the norm of the
space-discretized matrix. In order to solve such large-scale AaO linear systems efficiently, researchers prefer to use the preconditioned Krylov subspace solvers, which will be discussed in the next section.

\section{A PinT preconditioner and the spectra analysis}
\label{sec3}
In this section, we construct an efficient PinT preconditioner for the system \eqref{eq2.2}
and analyze the eigenvalue distribution of the preconditioned matrix.
The following assumption about the matrix $A$ is required in what follows.
\begin{assumption}
The matrix $A$ is negative semi-definite but diagonalizable. 
\label{assumption1}
\end{assumption}
It is worth noting that the matrix $A$ resulting from the spatial semi-discretization of the operator $\mathcal{L}$ in (1.2) is usually diagonalizable. Even when a direct proof is unavailable, diagonalizability remains a rather mild assumption (in Assumption 1), since non-diagonalizable matrices form a set of measure zero in $\mathbb{C}^{N_x \times N_x}$.

\subsection{The PinT preconditioner}
\label{sec3.1}
Replacing $B_1$ and $B_2$ by $\alpha$-circulant matrices
\begin{equation*}
C^{(\alpha)}_1 =
\begin{bmatrix}
	1 &  & & -\alpha \\
	-1 & 1 &  & \\
	& \ddots & \ddots & \\
	&  & -1 & 1
\end{bmatrix} \in \mathbb{R}^{N_t \times N_t}  \quad \mathrm{and} \quad
C^{(\alpha)}_2 = \frac{1}{2}
\begin{bmatrix}
	1 &  & & \alpha \\
	1 & 1 &  & \\
	& \ddots & \ddots & \\
	&  & 1 & 1
\end{bmatrix} \in \mathbb{R}^{N_t \times N_t},
\end{equation*}
respectively.
Here $\alpha \in (0,1)$ is a free parameter.
Then, we obtain our PinT preconditioner:
\begin{equation}\label{eq3.1}
\mathcal{P}_\alpha = C^{(\alpha)}_1 \otimes I_s - C^{(\alpha)}_2 \otimes \tilde{A}.
\end{equation}

Let $\mathbb{F} = \left[ \frac{\omega^{j k}}{\sqrt{N_t}}\right]_{j,k = 0}^{N_t - 1}$ be the discrete
Fourier matrix, where $\omega = e^{-2 \pi \mathrm{\mathbf{i}}/N_t},~\mathrm{\mathbf{i}} = \sqrt{-1}$.
We know that $\alpha$-circulant matrices $C^{(\alpha)}_{\ell}$ ($\ell = 1,2$) can be diagonalized as
$C^{(\alpha)}_{\ell} = V_\alpha D_{\ell} V_\alpha^{-1}$ with
\begin{equation*}
V_\alpha = \Lambda_\alpha^{-1} \mathbb{F}^* \quad \mathrm{and} \quad
D_{\ell} = \mathrm{diag}\left( \sqrt{N_t}\, \mathbb{F}\, \Lambda_\alpha\, C^{(\alpha)}_{\ell}(:,1) \right)
= \mathrm{diag}\left( \lambda_{\ell,1},\ldots,\lambda_{\ell,N_t} \right),
\end{equation*}
where $\Lambda_\alpha = \mathrm{diag}\left(1, \alpha^{\frac{1}{N_t}}, \ldots, \alpha^{\frac{N_t - 1}{N_t}} \right)$,
`*' denotes the conjugate transpose of a matrix,
$C^{(\alpha)}_{\ell}(:,1)$ is the first column of $C^{(\alpha)}_{\ell}$
and
\begin{equation}
	\lambda_{1,k} = 1 - \alpha^{\frac{1}{N_t}} e^{\frac{2 (k - 1) \pi \mathrm{\mathbf{i}}}{N_t}}, \quad
	\lambda_{2,k} = \frac{1}{2}  + \frac{1}{2} \alpha^{\frac{1}{N_t}}
	e^{\frac{2 (k - 1) \pi \mathrm{\mathbf{i}}}{N_t}}~(k = 1, \ldots, N_t).
\label{eigCx}
\end{equation}
Moreover, for $\alpha \in (0,1)$, the real part of $\lambda_{1,k}$ is positive, i.e.,
$\Re(\lambda_{1,k}) > 0$. 
Similarly, $\Re(\lambda_{2,k}) > 0$.

With these decompositions, we have
\begin{equation*}
\mathcal{P}_\alpha = \left( V_{\alpha} \otimes I_s \right)
\left( D_1 \otimes I_s - D_2 \otimes \tilde{A} \right)
\left( V_{\alpha}^{-1} \otimes I_s \right).
\end{equation*}
Thus, for a given vector $\bm{v}$, $\bm{z} = \mathcal{P}_\alpha^{-1} \bm{v}$ can be computed via the following
three steps:
\begin{equation}\label{eq3.2}
\begin{split}
& \textrm{step-(a)}\quad \bm{z}_1 = \left( V_{\alpha}^{-1} \otimes I_s \right) \bm{v}, \\
& \textrm{step-(b)}\quad \left( \lambda_{1,k} I_s - \lambda_{2,k} \tilde{A} \right) \bm{z}_{2,k} = \bm{z}_{1,k},~
k = 1, 2, \ldots, N_t, \\
& \textrm{step-(c)}\quad \bm{z} = \left( V_{\alpha} \otimes I_s \right) \bm{z}_2,
\end{split}
\end{equation}
where $\bm{z}_{\ell,k} = \bm{z}_{\ell}\left( (k - 1) N_x + 1 : k N_x \right)$ ($\ell = 1,2$)
represents the $k$-th block of $\bm{z}_{\ell}$.
Combining $\Re(\lambda_{\ell,k}) > 0$ ($\ell = 1,2$) and
Assumption \ref{assumption1},
we know that $\mathcal{P}_\alpha$ is nonsingular.

In \eqref{eq3.2}, step-(a) and step-(c) can be computed efficiently via FFTs (i.e., fast Fourier transforms).
In step-(b), the complex-shifted linear systems
can be solved by using some sparse direct solvers since $A$ is generally a sparse matrix.
For example, if $\mathcal{L} = \kappa \Delta$,
the fast discrete sine transform (DST) is used in step-(b).
\begin{remark}\label{remark1}
Similar to the tricks in \cite{lin2021all,gu2021note},
we find that the complex eigenvalue/eigenvector
of $C_{\ell}^{(\alpha)}$ ($\ell = 1,2$) appear in conjugate pairs.
Thus, we only need to solve first $\left \lceil \frac{N_t + 1}{2} \right \rceil$ systems in step-(b).
The remaining systems are solved by conjugating the obtained solutions in order to reduce the computational cost.
\end{remark}
\begin{remark}\label{remark1x}
For the preconditioner $\mathcal{P}_1$ (i.e., $\alpha = 1$), its invertibility is equivalent to the invertibilities of complex matrices 
$\lambda_{1,k}I_s - \lambda_{2,k}\tilde{A}~(k = 1,2,\cdots,N_t)$. However, when $\alpha = 1$, then it obviously follows that $\lambda_{1,1} = 0$ and $\lambda_{2,1} = 1$, cf.~Eq.~\eqref{eigCx}. If the space-discretization matrix $A$ of some particular PDEs is singular, then it means that the preconditioner $\mathcal{P}_1$ becomes singular and fails to be a valid preconditioner. Unfortunately, such a difficulty has not been discussed and shown in the existing literature \cite{McDonald18,lin2021all,Liu2020,WU2021}.
\end{remark}
\subsection{The spectral analysis}
\label{sec3.2}
Let
\begin{equation*}
	Q_1 = I_s - \frac{1}{2} \tilde{A} \quad \mathrm{and} \quad Q_2 = I_s + \frac{1}{2} \tilde{A}.
\end{equation*}
Then,
\begin{equation*}
	\mathcal{M} =
	\begin{bmatrix}
		Q_1  &  &  & \\
		-Q_2  & Q_1 &  & \\
		& \ddots & \ddots & \\
		&  & -Q_2 & Q_1
	\end{bmatrix} \quad \mathrm{and} \quad
\mathcal{P}_\alpha =
\begin{bmatrix}
	Q_1  &  &  & -\alpha Q_2\\
	-Q_2  & Q_1 &  & \\
	& \ddots & \ddots & \\
	&  & -Q_2 & Q_1
\end{bmatrix}.
\end{equation*}

From the above expressions, we can check that the matrix $\mathcal{M}$ is invertible since $Q_1$ is invertible.
With the help of these expressions, the following result is true.
\begin{lemma}
The preconditioned matrix
$\mathcal{P}_\alpha^{-1} \mathcal{M}$ has $(N_t - 1) N_x$ eigenvalues equal to $1$
and $N_x$ eigenvalues equal to those of a matrix inverse $\left(I_s - \alpha J_{N_t} \right)^{-1}$,
where $J_{N_t} = \left(Q_1^{-1} Q_2 \right)^{N_t}$.
\label{lemma3.1}
\end{lemma}
\begin{proof}
It is difficult to calculate $\mathcal{P}_\alpha^{-1} \mathcal{M}$ directly.
Thus, we first calculate $\mathcal{M}^{-1} \mathcal{P}_\alpha$.
We rewrite $\mathcal{M}$ and $\mathcal{P}_\alpha$ into the form
$\mathcal{M} =  \left(I_t \otimes Q_1 \right) \tilde{\mathcal{M}}$
and $\mathcal{P}_\alpha = \left(I_t \otimes Q_1 \right) \tilde{\mathcal{P}_\alpha}$,
where $I_t$ is an identity matrix of size $N_t$,
\begin{equation*}
\tilde{\mathcal{M}} =
\begin{bmatrix}
	I_s  &  &  & \\
	-Q_1^{-1} Q_2  & I_s &  & \\
	& \ddots & \ddots & \\
	&  & -Q_1^{-1} Q_2 & I_s
\end{bmatrix}, \quad
\tilde{\mathcal{P}_\alpha} =
\begin{bmatrix}
	I_s  &  &  & -\alpha Q_1^{-1} Q_2\\
	-Q_1^{-1} Q_2  & I_s &  & \\
	& \ddots & \ddots & \\
	&  & -Q_1^{-1} Q_2 & I_s
\end{bmatrix}.
\end{equation*}
By simple calculation, we have
\begin{equation*}
\tilde{\mathcal{M}}^{-1} =
\begin{bmatrix}
	I_s  &  &  & \\
	J_1  & I_s &  & \\
	\vdots & \ddots & \ddots & \\
	J_{N_t - 1} & \cdots & J_1 & I_s
\end{bmatrix},
\end{equation*}
where $J_k = \left(Q_1^{-1} Q_2 \right)^k$ ($k = 1, 2, \ldots$).
Then, we arrive at
\begin{equation}
\mathcal{M}^{-1} \mathcal{P}_\alpha = \tilde{\mathcal{M}}^{-1} \tilde{\mathcal{P}_\alpha} =
\begin{bmatrix}
	I_s & \bm{0} & \cdots  & \bm{0} & -\alpha J_1 \\
	& I_s & \ddots & \ddots  & -\alpha J_2 \\
	&  & \ddots & \ddots  & \vdots \\
	&  &  &  I_s & -\alpha J_{N_t - 1} \\
	&  &  &  & I_s - \alpha J_{N_t}
\end{bmatrix}.
\end{equation}
With this at hand, we obtain
\begin{equation*}
\mathcal{P}_\alpha^{-1} \mathcal{M} = \left( \mathcal{M}^{-1} \mathcal{P}_\alpha \right)^{-1}
= \left( \tilde{\mathcal{M}}^{-1} \tilde{\mathcal{P}_\alpha} \right)^{-1} =
\begin{bmatrix}
	I_s & \bm{0} & \cdots  & \bm{0} & \alpha J_1 \left(I_s - \alpha J_{N_t}\right)^{-1} \\
	& I_s & \ddots & \ddots  & \alpha J_2 \left(I_s - \alpha J_{N_t}\right)^{-1} \\
	&  & \ddots & \ddots  & \vdots \\
	&  &  &  I_s & \alpha J_{N_t - 1} \left(I_s - \alpha J_{N_t}\right)^{-1} \\
	&  &  &  & \left(I_s - \alpha J_{N_t}\right)^{-1}
\end{bmatrix}.
\label{P^-1M}
\end{equation*}	
The proof is completed.
\end{proof}
It is worth mentioning that this proof is different from \cite{McDonald18}
since the Sherman-Morrison-Woodbury formula fails in our case.
The reason is that the matrix $Q_2$ may be singular.
Recall the proof of Lemma \ref{lemma3.1}, the invertibility of $I_s - \alpha J_{N_t}$
is not proved. Fortunately, it can be obtained from the following lemma.
\begin{lemma}
The eigenvalues of $Q_1^{-1} Q_2$ satisfy
\begin{equation*}
	-1 < \Re \left( \lambda \left(Q_1^{-1} Q_2 \right) \right) \leq 1 \quad
	\mathrm{and}\quad
	0 \leq \left| \lambda \left(Q_1^{-1} Q_2 \right) \right| \leq 1.
\end{equation*}
\label{lemma3.2}
\vspace{-2.5em}
\end{lemma}
\begin{proof}
We know that all the eigenvalues of $A$ can be written in the form
$\lambda \left(A \right) = \eta + \xi \mathrm{\mathbf{i}}$.
Noticing Assumption \ref{assumption1}, we have $\eta,\xi \in \mathbb{R}$ and $\eta \leq 0$.
Then, it is not difficult to get that
\begin{equation*}
\lambda \left(Q_1^{-1} Q_2 \right) 
= \frac{\lambda \left(Q_2 \right)}{\lambda \left(Q_1 \right)}
=\frac{1 + \frac{\tau}{2} \lambda \left(A \right)}{1 - \frac{\tau}{2} \lambda \left(A \right)}
= \frac{1 - \left(\frac{\tau}{2}\eta\right)^2 - \left(\frac{\tau}{2}\xi\right)^2}
{\left(1 - \frac{\tau}{2}\eta \right)^2 + \left(\frac{\tau}{2}\xi \right)^2}
+ \frac{\tau \xi}
{\left(1 - \frac{\tau}{2}\eta \right)^2 + \left(\frac{\tau}{2}\xi \right)^2} \mathrm{\mathbf{i}}.
\end{equation*}
We first estimate the real part of it. Obviously,
\begin{equation*}
\Re \left( \lambda \left(Q_1^{-1} Q_2 \right) \right)
=\frac{1 - \left(\frac{\tau}{2}\eta\right)^2 - \left(\frac{\tau}{2}\xi\right)^2}
{\left(1 - \frac{\tau}{2}\eta \right)^2 + \left(\frac{\tau}{2}\xi \right)^2}
= -1 + 	\frac{2 - \tau \eta}
{\left(1 - \frac{\tau}{2}\eta \right)^2 + \left(\frac{\tau}{2}\xi \right)^2} > -1.	
\end{equation*}
On the other hand, we obtain
\begin{equation*}
\left(1 - \frac{\tau}{2}\eta \right)^2 + \left(\frac{\tau}{2}\xi \right)^2 -
\left[ 1 - \left(\frac{\tau}{2}\eta\right)^2 - \left(\frac{\tau}{2}\xi\right)^2 \right]
= -\tau \eta + \frac{\tau^2}{2}\eta^2 + \frac{\tau^2}{2}\xi^2 \geq 0.
\end{equation*}
Thus, we arrived at
\begin{equation*}
-1 < \Re\left( \lambda \left(Q_1^{-1} Q_2 \right) \right) \leq 1.
\end{equation*}
Next, we turn to show the upper bound of $\left| \lambda \left(Q_1^{-1} Q_2 \right) \right| $.
By some calculations, we have
\begin{equation*}
\left[ 1 - \left(\frac{\tau}{2}\eta \right)^2 - \left(\frac{\tau}{2} \xi \right)^2 \right]^2
+ \left(\tau \xi \right)^2 - \left[ \left(1 - \frac{\tau}{2} \eta \right)^2 + \left(\frac{\tau}{2} \xi \right)^2 \right]^2 
=2 \tau \eta \left(\frac{\tau^2}{4}\eta^2 + \frac{\tau^2}{4} \xi^2  - \tau \eta + 1\right)
\leq 0.
\end{equation*}
This implies that $\left| \lambda \left(Q_1^{-1} Q_2 \right) \right| \leq 1$.
\end{proof}
With the help of Lemma \ref{lemma3.2},
we can immediately show the nonsingularity of $\left(I_s - \alpha J_{N_t}\right)$.
\begin{lemma}
Assume $\alpha \in (0,1)$. The matrix $\left(I_s - \alpha J_{N_t}\right)$ is nonsingular
and its eigenvalues satisfy
\begin{equation*}
1 - \alpha\leq \Re\left( \lambda \left( I_s - \alpha J_{N_t} \right) \right) \leq 1+\alpha
\quad \mathrm{and}\quad
1 - \alpha \leq \left| \lambda \left(I_s - \alpha J_{N_t} \right) \right| \leq 1 + \alpha.
\end{equation*}
\label{lemma3.3}
\vspace{-2em}
\end{lemma}
\begin{proof}
The eigenvalues of $\left(I_s - \alpha J_{N_t}\right)$ can be expressed as
\begin{equation*}
\lambda \left(I_s - \alpha J_{N_t} \right) =
1 - \alpha \left[\lambda \left( Q_1^{-1} Q_2 \right)\right]^{N_t}.
\end{equation*}
Based on Lemma \ref{lemma3.2}, we rewrite the (complex) eigenvalues of $Q_1^{-1} Q_2$ in the polar form:
$\lambda \left( Q_1^{-1} Q_2 \right) = r \left(\cos \theta + \mathrm{\mathbf{i}}  \sin \theta \right)$,
where $0 \leq r \leq 1$ and $\theta \in (-\pi, \pi)$.
With the help of De Moivre's formula, we have
\begin{equation*}
0 < 1 - \alpha \leq \Re\left( \lambda \left( I_s - \alpha J_{N_t} \right) \right)
= 1 - \alpha r^{N_t} \cos \left(N_t \theta \right) \leq 1 + \alpha.
\end{equation*}
This also implies that the matrix $\left(I_s - \alpha J_{N_t}\right)$ is nonsingular.

Using Lemma \ref{lemma3.2}, we obtain
\begin{equation*}
1 - \alpha \leq \left| \Re\left( \lambda \left( I_s - \alpha J_{N_t} \right) \right) \right|
\leq \left| \lambda \left(I_s - \alpha J_{N_t} \right) \right|
\leq 1 + \alpha \left| \lambda \left(Q_1^{-1} Q_2 \right) \right|^{N_t} \leq 1 + \alpha.
\end{equation*}
This completes the proof.
\end{proof}
Combining Lemma \ref{lemma3.1} and Lemma \ref{lemma3.3},
we get the spectral distribution of $\mathcal{P}_\alpha^{-1} \mathcal{M}$.
\begin{theorem}
\label{thm3.1}
Assume $\alpha \in (0,1)$ and the spatial matrix $A$ fulfills Assumption \ref{assumption1}.
Then, the preconditioned matrix
$\mathcal{P}_\alpha^{-1} \mathcal{M}$ has $(N_t - 1) N_x$ eigenvalues equal to $1$
and $N_x$ eigenvalues distributed in the annulus
\begin{equation*}
\Omega_{\alpha} = \left\{ z \in \mathbb{C}: \frac{1}{1 + \alpha} \leq \left| z \right|
\leq \frac{1}{1 - \alpha}~ \mathrm{and}~ 
\Re\left(z \right) >0
\right\}.
\end{equation*}
\label{th3.1}
\end{theorem}
\begin{proof}
Lemma \ref{lemma3.1} implies 
\begin{equation*} 
\sigma( \mathcal{P}_\alpha^{-1} \mathcal{M} ) 
=\underbrace{\{1,1,\ldots,1\}}_{(N_t -1)N_x} \cup \, \sigma(\left(I_s - \alpha J_{N_t} \right)^{-1}),
\end{equation*}
so  we only need to show the distribution of the eigenvalues of
$\left(I_s - \alpha J_{N_t} \right)^{-1}$.

From Lemma \ref{lemma3.3}, we know
\begin{equation*}
\frac{1}{1 + \alpha} \leq \left| \lambda \left( \left(I_s - \alpha J_{N_t} \right)^{-1} \right) \right|
	\leq \frac{1}{1 - \alpha}
\end{equation*}
and $\Re\left(\lambda \left( \left(I_s - \alpha J_{N_t} \right)^{-1}\right) \right)>0$,
which completes the proof.
\end{proof}

Moreover, if we tighten Assumption \ref{assumption1} as that the spatial matrix $A$ is symmetric negative semi-definite, then we can improve the above theorem as follows.
\begin{corollary}
If the spatial matrix $A$ is symmetric negative semi-definite and $\alpha\in(0,1)$, then the preconditioned matrix
$\mathcal{P}_\alpha^{-1} \mathcal{M}$ has $(N_t - 1) N_x$ eigenvalues equal to $1$ and $N_x$ eigenvalues distributed in
the interval $\left[\frac{1}{1+\alpha}, \frac{1}{1-\alpha}\right]$.
\end{corollary}
\begin{proof}
Since the spatial matrix $A$ is symmetric negative semi-definite, both the matrices $Q_1$ and $Q_2$ are
real symmetric, then $I_s - \alpha J_{N_t}$ will be real symmetric and all its (real) eigenvalues will be located at
the interval $[1-\alpha,1+\alpha]$ according to Lemma \ref{lemma3.2} and Lemma \ref{lemma3.3}. Therefore, each eigenvalue $z\in\mathbb{R}$ of
the matrix $(I_s - \alpha J_{N_t})^{-1}$ satisfies that $z\in\left[\frac{1}{1+\alpha}, \frac{1}{1-\alpha}\right]$,
then we complete the proof with the help of Theorem \ref{th3.1}.
\end{proof}
\begin{remark}
Compared with the stability condition required by the Crank-Nicolson scheme (cf.~Theorem \ref{the2.1}),
we only need to assume that the spatial matrix $A$ can be diagonalized to give our eigenvalue analysis
for the preconditioned matrix. In Theorem \ref{th3.1}, the upper bound can be reached, if the spatial
matrix $A$ has some zero eigenvalues; refer to Section \ref{sec5} for a discussion.
\label{remark2x}
\end{remark}
%
%
%

\section{Convergence rate of the preconditioned GMRES method}
\label{sec4}
Let ${\bm u}_k$ denote the approximate solution at the $k$-th iteration of GMRES applied to the
preconditioned system $\mathcal{P}^{-1}_{\alpha}\mathcal{M}{\bm u}_k = \mathcal{P}^{-1}_{\alpha}
{\bm b}$ with a given initial guess ${\bm u}_0$. Then the residual vector ${\bm r}_k:= \mathcal{
P}^{-1}_{\alpha} ({\bm b} - \mathcal{M}{\bm u}_k)$ satisfies the minimal residual property \cite{saad2003iterative}
\begin{equation}
\|{\bm r}_k\|_2 = \min_{p\in \mathbb{P}_k, p(0)=1}\|p(\mathcal{P}^{-1}_{\alpha}\mathcal{M})
{\bm r}_0\|_2\leq \left(\min_{p\in \mathbb{P}_k, p(0)=1}\| p(\mathcal{P}^{-1}_{\alpha} \mathcal{M})\|_2\right)
\|{\bm r}_0\|_2,
\label{sss2x}
\end{equation}
where $\mathbb{P}_k$ is the set of polynomials of degree $k$ or less. For a given initial ${\bm r}_0$,
the last ``worst-case" inequality can be highly overestimated in some cases \cite{titley2014,embree2022}. It is usually
difficult to directly estimate the norm $\|p(\mathcal{P}^{-1}_{\alpha} \mathcal{M})\|_2$ from only
the spectral information of the nonnormal matrix $\mathcal{P}^{-1}_{\alpha}\mathcal{M}$. However, if
$\mathcal{P}^{-1}_{\alpha}\mathcal{M}$ is diagonalizable, then the above estimate (\ref{sss2x}) can be
improved. In fact, we have the following conclusion,
\begin{theorem}
\label{thm4.1}
Assume $\alpha\in(0,1)$ and the spatial matrix $A$ fulfills Assumption \ref{assumption1}. The preconditioned matrix $\mathcal{P}^{-1}_{\alpha}\mathcal{M}$ is diagonalizable for any $\alpha \in (0,1)$.
\end{theorem}
\begin{proof}
We first claim that the matrix $I_s - \alpha J_{N_t}$ is diagonalizable. Due to Assumption \ref{assumption1},
the eigendecomposition of $\tilde{A}$ is given as $\tilde{A} = V_aD_aV^{-1}_a$ and then
\begin{equation}
Q^{-1}_1Q_2 = V_a\left(I_s - \frac{1}{2}D_a\right)^{-1}V^{-1}_aV_a\left(I_s + \frac{1}{2}D_a\right)V^{-1}_a:=V_aD_{J}V^{-1}_a
\label{eq4.2x}
\end{equation}
with $D_J := \left(I_s - \frac{1}{2}D_a\right)^{-1}\left(I_s + \frac{1}{2}D_a\right)$ being a diagonal
matrix whose diagonal entries are the eigenvalues of $Q^{-1}_1Q_2$. Moreover, we have the following matrix diagonalization
\begin{equation}
\begin{split}
I_s - \alpha J_{N_t} & = I_s - \alpha (Q^{-1}_1Q_2)^{N_t}
= V_a[I_s - \alpha (D_{J})^{N_t}]V^{-1}_a
\end{split}
\end{equation}
and $(I_s - \alpha J_{N_t})^{-1} - I_s $ can be factorized as
\begin{equation}
(I_s - \alpha J_{N_t})^{-1} - I_s = V_a[(I_s - \alpha (D_{J})^{N_t})^{-1} - I_s]V^{-1}_a:=V_a\Theta V^{-1}_a.
\end{equation}
By the above expression of $ \mathcal{P}_{\alpha}^{-1} \mathcal{M} $ in \eqref{P^-1M}, a tedious but simple calculation yields the following diagonalization:
\begin{equation*}
\mathcal{P}_\alpha^{-1} \mathcal{M} - \mathcal{I} =
\underbrace{\begin{bmatrix}
	I_s & \bm{0} & \cdots  & \bm{0} & \alpha J_1V_a  \\
	& I_s & \ddots & \ddots  & \alpha J_2V_a  \\
	&  & \ddots & \ddots  & \vdots \\
	&  &  &  I_s & \alpha J_{N_t - 1}V_a \\
	&  &  &  & V_a
\end{bmatrix}}_{\mathcal{S}}\underbrace{\begin{bmatrix}
	\bm{0}   \\
	& \bm{0}  \\
	&  & \ddots  \\
	&  &  &  \bm{0}  \\
	&  &  &  & \Theta
\end{bmatrix}}_{\widetilde{\Theta}}\underbrace{\begin{bmatrix}
	I_s & \bm{0} & \cdots  & \bm{0} & \alpha J_1 V_a\\
	& I_s & \ddots & \ddots  & \alpha J_2 V_a \\
	&  & \ddots & \ddots  & \vdots \\
	&  &  &  I_s & \alpha J_{N_t - 1} V_a\\
	&  &  &  & V_a
\end{bmatrix}^{-1}}_{\mathcal{S}^{-1}},
\end{equation*}	
which leads to the desired diagonalization
\begin{equation}
\mathcal{P}_\alpha^{-1} \mathcal{M}  = \mathcal{S}(\mathcal{I} + \widetilde{\Theta})\mathcal{S}^{-1}:=\mathcal{S}\Upsilon\mathcal{S}^{-1},
\end{equation}
where the diagonal entries of $\Upsilon$ consist of the spectrum set $\sigma(\mathcal{P}^{-1}_{\alpha}\mathcal{M})$.
\end{proof}

According to the above theorem, it is natural to obtain a convenient upper bound
\begin{equation}
\begin{split}
\|p(\mathcal{P}^{-1}_{\alpha}\mathcal{M})\|_2 & =
\|\mathcal{S}p(\Upsilon)\mathcal{S}^{-1}\|_2 \\
& \leq \|\mathcal{S}\|_2\|p(\Upsilon)\|_2\|\mathcal{S}^{-1}\|_2 \\
& = \kappa(\mathcal{S})\left(\max_{\lambda \in \sigma(\mathcal{P}^{-1}_{\alpha} \mathcal{M})}|p(\lambda)|\right),
\end{split}
\end{equation}
where $\kappa(\mathcal{S}):=\|\mathcal{S}\|_2\|\mathcal{S}^{-1}\|_2$
denotes the 2-norm condition number of eigenvector matrix $\mathcal{S}$
and $\sigma(\mathcal{P}^{-1}_{\alpha}\mathcal{M})$ is explicitly given in
Theorem \ref{th3.1}. In particular, this bound indicates that a highly clustered
spectrum away from zero will likely give a fast convergence rate if $\kappa(\mathcal{S})$ is not too large. Following the arguments in \cite[Theorem 3]{McDonald18},
we have already proven that $\mathcal{P}^{-1}_{\alpha}\mathcal{M}$ is diagonalizable,
and hence the above convergence bound is indeed applicable. Nevertheless, it is
still difficult to estimate $\kappa(\mathcal{S})$, especially when
we just make the spatial matrix $A$ fulfill Assumption \ref{assumption1}. This phenomenon motivates us to
use an alternative approach to obtain a different convergence rate bound, as explained
in the next section.

In fact, for any matrix $Z\in\mathbb{R}^{n\times n}$, denote $\mathcal{H}(Z) = \frac{Z + Z^{\top}}{2}$ as
the symmetric part of $Z$. According to the convergence bound of preconditioned restarted GMRES \cite{Eisenstat83,Elman1982phd}
for nonsymmetric but positive definite systems, we have the following relative residual norm
estimate:
\begin{equation}
\frac{\|{\bm r}_m^{(k)} \|_2}{\|{\bm r}_0\|_2}\leq
\left[1 - \frac{\lambda^{2}_{\min}(\mathcal{H}(\mathcal{P}^{-1}_{\alpha}\mathcal{M}))}
{\|\mathcal{P}^{-1}_{\alpha}\mathcal{M}\|^{2}_2}\right]^{km/2},
\label{eq4.7}
\end{equation}
where $m$ indicates that GMRES is restarted after every $m$ iterations, and ${\bm r}_m^{(k)}$ is the residual vector after $k$-th restart of the GMRES($m$).
We highlight that such a (possibly nonsharp) convergence rate requires the symmetric part
of the preconditioned matrix (i.e., $\mathcal{H}(\mathcal{P}^{-1}_{\alpha}\mathcal{M})$) to be
positive definite, which will be shown in the next context.

At this stage, since
\begin{equation}
\mathcal{P}_{\alpha} = \mathcal{M} - \alpha\begin{bmatrix}
	\bm{0}   & \cdots&  \bm{0}& Q_2\\
	 & \bm{0} & \cdots & \bm{0}\\
	&  & \ddots & \vdots\\
	&  &  & \bm{0}
\end{bmatrix}:=\mathcal{M} - \alpha\mathcal{R},
\end{equation}
then we have
\begin{equation}
\|\mathcal{P}^{-1}_{\alpha}\mathcal{M}\|_2 = \|\mathcal{I} + \alpha\mathcal{P}^{-1}_{\alpha}\mathcal{R}\|_2
\leq 1 + \alpha\|\mathcal{P}^{-1}_{\alpha}\mathcal{M}\|_2\|\mathcal{M}^{-1}\mathcal{R}\|_2.
\label{eq4.9}
\end{equation}
The application of \eqref{eq4.7} hinges on estimating 
$ \lambda_{\min}(\mathcal{H}(\mathcal{P}^{-1}_{\alpha}\mathcal{M})) $ 
and $ \|\mathcal{P}^{-1}_{\alpha}\mathcal{M}\|_2 $. To achieve this, we first estimate the upper bound of the norm $\|\mathcal{M}^{-1}\mathcal{R}\|_2$. Unfortunately, if the spatial matrix $A$ just fulfills Assumption \ref{assumption1},
we cannot estimate the upper bound of the norm $\|\mathcal{M}^{-1}\mathcal{R}\|_2$ well,
because the useful property of matrix spectral norm which is invariant under orthogonal transformations
will be unavailable. 
Next, we assume that the spatial matrix $A$ is normal and fulfills Assumption \ref{assumption1}. 
Then, $A$ admits a unitary eigendecomposition $\tilde{A} = V_a D_a V_a^*$, 
where $V_a$ is a unitary matrix (i.e., $V_a^* = V_a^{-1}$) and $D_a$ is diagonal. 
Consequently, $J_1 = Q_1^{-1} Q_2 = V_a D_J V_a^*$ according to Eq. \eqref{eq4.2x} and 
$D_J = {\rm diag}(\lambda_1,\lambda_2,\ldots,\lambda_{N_x})$ is the diagonal matrix of eigenvalues
of $Q^{-1}_1Q_2$ satisfying $|\lambda_j|\leq 1$, $ j=1,2,\ldots,N_x$, cf.~Lemma \ref{lemma3.2}. 
\begin{lemma}
Assuming that the spatial matrix $A$ is normal and fulfills Assumption \ref{assumption1}, then
\begin{equation*}
\|\mathcal{M}^{-1}\mathcal{R}\|_2 \leq \sqrt{N_t}.
\end{equation*}
\label{lemma4.1}
\vspace{-2.5em}
\end{lemma}
\begin{proof} We compute
\begin{equation*}
\begin{split}
\mathcal{M}^{-1}\mathcal{R} &=
\begin{bmatrix}
	I_s  &  &  & \\
	J_1  & I_s &  & \\
	\vdots & \ddots & \ddots & \\
	J_{N_t - 1} & \cdots & J_1 & I_s
\end{bmatrix}
\left( I_t \otimes Q_1^{-1} \right)
\begin{bmatrix}
	\bm{0}   & \cdots&  \bm{0}& Q_2\\
	 & \bm{0} & \cdots & \bm{0}\\
	&  & \ddots & \vdots\\
	&  &  & \bm{0}
\end{bmatrix} \\
& = \begin{bmatrix}
	I_s  &  &  & \\
	J_1  & I_s &  & \\
	\vdots & \ddots & \ddots & \\
	J_{N_t - 1} & \cdots & J_1 & I_s
\end{bmatrix} 
\begin{bmatrix}
	\bm{0}   & \cdots&  \bm{0} & J_1 \\
	 & \bm{0} & \cdots & \bm{0}\\
	&  & \ddots & \vdots\\
	&  &  & \bm{0}
\end{bmatrix}
= \begin{bmatrix}
	\bm{0}   & \cdots&  \bm{0}& J_1\\
	 & \bm{0} & \cdots & J_2\\
	&  & \ddots & \vdots\\
	&  &  & J_{N_t}
\end{bmatrix}.
\end{split}
\end{equation*}
Then, we have $J_n^* J_n = V_a |D_J|^{2n} V_a^*$, where $V_a$ is a unitary matrix ($V_a^* V_a = I$) 
and $|D_J| = {\rm diag}\big(|\lambda_1|, |\lambda_2|, \ldots,\\ |\lambda_{N_x}|\big)$. 
At this stage, one can obtain
\begin{align*}
\|\mathcal{M}^{-1}\mathcal{R}\|_2^2 &= \lambda_{\max}\left( (\mathcal{M}^{-1}\mathcal{R})^* \mathcal{M}^{-1}\mathcal{R} \right) 
= \lambda_{\max}\left( \sum_{n=1}^{N_t} J_n^* J_n \right) \\
&\le \left\| \sum_{n=1}^{N_t} |D_J|^{2n} \right\|_2 
= \max_{1 \le j \le N_x} \sum_{n=1}^{N_t} |\lambda_j|^{2n} \le N_t
\end{align*}
due to $|\lambda_j| \le 1$ (cf.~Lemma \ref{lemma3.2}) and thus $\|\mathcal{M}^{-1}\mathcal{R}\|_2 \le \sqrt{N_t}$.
\end{proof}

In view of (\ref{eq4.7}) and the above analysis, we can arrive at the following linear
convergence rate of the preconditioned GMRES($m$) method.
\begin{theorem}
\label{them4.2}
Assume that the spatial matrix $A$ is normal and fulfills Assumption \ref{assumption1}. For any given positive constant $\delta \in (0,1/2)$, the  left-preconditioned restarted GMRES for solving $\mathcal{M}{\bm u} = {\bm b}$ with the preconditioner $\mathcal{P}_{\alpha}$ achieves the following mesh-independent linear convergence rate:
\begin{equation}
\frac{\|{\bm r}_m^{(k)} \|_2}{\|{\bm r}_0\|_2}
\leq \left[ 1 - (1 - 2\delta)^2\right]^{km/2} 
= \left[ 2\sqrt{\delta (1 - \delta)}\right]^{km},
\quad  \forall \alpha \in (0,\alpha_0),
\end{equation}
where $ \alpha_0 = \delta / \sqrt{N_t} = \delta\sqrt{\tau/T} $. 
\end{theorem}
\begin{proof}
When $\alpha \in (0,\alpha_0)$, we have $\alpha \|\mathcal{M}^{-1} \mathcal{R}\|_2  \leq \alpha \sqrt{N_t} \leq \delta  <1/2 $ according to Lemma \ref{lemma4.1}. Combining \eqref{eq4.9}, it holds that
\begin{equation}
\|\mathcal{P}^{-1}_{\alpha}\mathcal{M}\|_2\leq \frac{1}{1 - \alpha\|\mathcal{M}^{-1}\mathcal{R}\|_2}.
\label{eq4.15}
\end{equation}
Similar to the conclusion of Ref. \cite{Liu2020}, we have
$\|\mathcal{P}^{-1}_{\alpha}\mathcal{M} - \mathcal{I}\|_2 \leq \frac{\alpha\|\mathcal{M}^{-1}\mathcal{R}\|_2}
{1 - \alpha \|\mathcal{M}^{-1}\mathcal{R}\|_2}$ and
\begin{equation}
\begin{split}
\lambda_{\min}(\mathcal{H}(\mathcal{P}^{-1}_{\alpha}\mathcal{M})) &= 1 +
\lambda_{\min}(\mathcal{H}(\mathcal{P}^{-1}_{\alpha}\mathcal{M} - \mathcal{I}))
\geq 1 - \|\mathcal{H}(\mathcal{P}^{-1}_{\alpha}\mathcal{M} - \mathcal{I})\|_2 \\
&\geq 1 - \|\mathcal{P}^{-1}_{\alpha}\mathcal{M} - \mathcal{I}\|_2 > 1 -
\frac{\alpha\|\mathcal{M}^{-1}\mathcal{R}\|_2}
{1 - \alpha \|\mathcal{M}^{-1}\mathcal{R}\|_2}>0.
\end{split}
\label{eq4.11}
\end{equation}
due to the well-known inequalities $|\lambda_{\min}(\mathcal{H}(A))|\leq
\|\mathcal{H(A)}\|_2 \leq \|A\|_2$ \cite[p.~160]{Alfio2007} and it also implies that
the symmetric part of the preconditioned matrix $\mathcal{P}^{-1}_{\alpha}
\mathcal{M}$ is positive definite at the moment. Hence, combining \eqref{eq4.15} and
\eqref{eq4.11} we obtained the following GMRES convergence rate estimate
\begin{equation}
\begin{split}
\frac{\|{\bm r}_m^{(k)} \|_2}{\|{\bm r}_0\|_2}
&< \left[1 - \left(\frac{1 -\frac{\alpha\|\mathcal{M}^{-1}\mathcal{R}\|_2}{1 - \alpha \|\mathcal{M}^{-1}\mathcal{R}\|_2}}{\frac{1}{1 - \alpha\|\mathcal{M}^{-1}\mathcal{R}\|_2}}\right)^2\right]^{km/2}
= \left[1 - (1 - 2\alpha\|\mathcal{M}^{-1}\mathcal{R}\|_2)^2\right]^{km/2} \\
&\leq \left[ 1 - (1 - 2\delta)^2\right]^{km/2} 
= \left[ 2\sqrt{\delta (1 - \delta)}\right]^{km}.
\end{split}
\end{equation}
This completes the proof.
\end{proof}
The above theorem indicates that the iteration counts of the preconditioned GMRES($m$) iterations will be mesh-independent.
It also shows that a smaller $\alpha$ leads to a faster convergence rate,
as indeed observed in numerical experiments. Among our tested examples, we observe that a fixed $\alpha = 0.1$
or $0.01$ already provides very fast mesh-independent convergence rate.  However, the theoretical condition
$\alpha = \mathcal{O}(\sqrt{\tau})$ is rather restrictive, because it implies that $\alpha$ goes to $0$ rapidly
as we refine the time step-size $\tau$.  In practical computation, the parameter $\alpha$ cannot be arbitrarily
small, because a too small $\alpha$ results in unacceptable roundoff errors in Step-(a) and Step-(c) of the
diagonalization procedure (\ref{eq3.2}), which may seriously pollute the obtained numerical solution;  see \cite{embree2023,Wu2022} for a discussion of these issues.
\begin{remark}[Applicability to non-normal matrices]
It should be noted that while the explicit {\rm GMRES($m$)} convergence rate in Theorem \ref{them4.2} relies on the normality of $A$ to establish the bound $\|\mathcal{M}^{-1}\mathcal{R}\|_2 \leq \sqrt{N_t}$ without involving the eigenvector condition number $\kappa(V_a)$, the spectral clustering result in Theorem \ref{thm3.1} and the diagonalizability in Theorem \ref{thm4.1} require only Assumption \ref{assumption1} (where $A$ need not be normal or symmetric). In practical applications, such as the non-normal spatial discretization matrices arising from option pricing PDEs, it is precisely this tight spectral clustering combined with the mild non-normality of $A$ that drives the fast, mesh-independent convergence observed in numerical experiments.
\end{remark}
\begin{remark}
In the next section, we observe that the number of GMRES iterations may not reach the preset restarting parameter $m$. However, this does not undermine the adoption of the restarted GMRES($m$) algorithm. The standard GMRES method requires the storage of all generated orthogonal basis vectors of the Krylov subspace, and the storage requirement grows linearly with the number of iterations. When dealing with large-scale problems or when convergence is slow, this can become a significant limitation due to excessive memory demand. In contrast, the restarted GMRES bounds the storage requirement to $\mathcal{O}(m N_t N_x)$. This substantially reduces memory pressure and provides robustness against unforeseen convergence delays, ensuring the scalability of the algorithm when applied to more challenging problems.
\end{remark}
\section{Numerical experiments}
\label{sec5}

In this section, we present several numerical experiments to assess the convergence behaviour and computational efficiency of the proposed generalized block $\alpha$-circulant preconditioner. The experiments are designed to validate the theoretical results established in the previous sections and to illustrate the practical performance of the proposed parallel-in-time solver.

All simulations were carried out in MATLAB R2023a on an ASUSTeK workstation equipped with a 13th Gen Intel(R) Core(TM) i9-13900K processor (3.00 GHz) and 128 GB RAM. CPU times (in seconds) were measured using MATLAB's {\tt tic}/{\tt toc} timing functions based on a partially parallel implementation of the proposed preconditioned iterative algorithms\footnote{The MATLAB implementation employs {\tt parfor} to exploit the available shared-memory parallelism.}. In the present paper, we focus on the convergence behaviour of the proposed algorithms. A comprehensive study of their performance on large-scale parallel architectures will be reported in future work.

Although our theoretical analysis is based on the left-preconditioned GMRES($m$) method, the right-preconditioned variant is preferable in practice since its stopping criterion is based on the true residual rather than on the preconditioned residual, which may differ substantially from the true residual \cite{saad2003iterative}. Throughout this section, we therefore employ the right-preconditioned GMRES(40) solver available in the $i$FEM package \cite{Chen:2008ifem}. Unless stated otherwise, a zero initial guess and a stopping tolerance of ${\tt tol}=10^{-9}$ based on the relative residual norm are used. The abbreviation ``Its'' denotes the total number of GMRES inner iterations (equivalently, matrix-vector products), while ``DoFs'' refers to the number of spatial degrees of freedom.

\subsection{Heston model}\label{sec5.1}

We first consider the Heston stochastic volatility model, governed by the PDE

\begin{equation}
\frac{\partial u}{\partial t}
=
\frac{1}{2}vs^2\frac{\partial^2u}{\partial s^2}
+
\sigma sv\rho\frac{\partial^2u}{\partial s\partial v}
+
\frac{1}{2}\sigma^2v\frac{\partial^2u}{\partial v^2}
+
rs\frac{\partial u}{\partial s}
+
\kappa(\eta-v)\frac{\partial u}{\partial v}
-
ru,
\label{eq5.1}
\end{equation}
where $u(s,v,t)$ denotes the price of a European option with underlying asset price $s$ and instantaneous variance $v$ at time $T-t^{*}$. The model is defined on the unbounded domain
\[
0\le t\le T,\qquad s>0,\qquad v>0,
\]
where $T$ denotes the maturity. The parameters satisfy $\kappa>0$, $\sigma>0$ and $\rho\in[-1,1]$. Throughout this example we assume the Feller condition
\[
2\kappa\eta>\sigma^2,
\]
which guarantees that the variance process remains strictly positive.
The initial condition is given by
\begin{equation}\label{eq5.2}
u(s,v,0)=\max\{0,s-K\},
\end{equation}
where $K>0$ denotes the strike price. The computational domain is truncated to the bounded region
\[
[0,S_{\max}]\times[0,V_{\max}],
\]
where $S_{\max}$ and $V_{\max}$ are chosen sufficiently large to minimize the influence of the artificial boundaries. In the present experiments we set $S_{\max}=800$ and $V_{\max}=4$. For a European call option, the boundary conditions are
\begin{eqnarray}
u(0,v,t)&=&0,\label{eq5.3}\\
u(s,V_{\max},t)&=&s,\label{eq5.4}\\
\frac{\partial u}{\partial s}(S_{\max},v,t)&=&1.\label{eq5.5}
\end{eqnarray}
The remaining boundary condition corresponds to the degenerate boundary at $v=0$ and is given by
\begin{equation*}
\frac{\partial u}{\partial t}(s,0,t)
=
\kappa\eta\frac{\partial u}{\partial v}(s,0,t)
+
rs\frac{\partial u}{\partial s}(s,0,t)
-
ru(s,0,t),
\end{equation*}
which serves as the artificial boundary condition for the Heston PDE \eqref{eq5.1}. After spatial discretization of Eqs.~\eqref{eq5.1}-\eqref{eq5.5}, with 
$N_s$ and $N_v$ denoting the number of grid nodes along the $s$- and $v$-directions, respectively, the resulting AaO linear system given in \eqref{eq2.2} is solved using the proposed preconditioned iterative method.

To assess the accuracy of the numerical solutions, we define the relative error by
\[
{\rm Err}
=
\frac{\left|{\rm Price}(S_0,V_0)-{\rm Price}_{\rm ref}\right|}
{\left|{\rm Price}_{\rm ref}\right|},
\]
where the reference value ${\rm Price}_{\rm ref}$ is computed using the COS method \cite{Fang09}. If the point $(S_0,V_0)$ does not coincide with a grid point, the corresponding option value is obtained by bilinear interpolation; see \cite[p.~68]{Graaf} for further details.

The numerical experiments are carried out for the following two sets of Heston model parameters \cite{Hout2010}:
\begin{itemize}
\item Set \uppercase\expandafter{\romannumeral1}: $T = 1, K = 100, \kappa = 1.5, \eta = 0.04, \sigma = 0.3, r = 0.025, \rho = -0.9, S_0 = 100, V_0 = 0.5$;
\item Set \uppercase\expandafter{\romannumeral2}: $T = 3, K = 100, \kappa = 0.6067, \eta = 0.0707, \sigma = 0.2928, r = 0.03, \rho = -0.7571, S_0 = 100, V_0 = 0.2$.
\end{itemize}
\begin{table}[ht]\tabcolsep=5.0pt
	\caption{Convergence of different right-preconditioned GMRES(40) methods with $N_t = N_s = 2N_v$ and $\alpha = 10^{-3}$ (Set \uppercase\expandafter{\romannumeral1}).}
	\centering
	\begin{tabular}{ccccccccc}
		\toprule
            Mesh   &$N_t$ &DoFs        &\multicolumn{3}{c}{$\mathcal{P}_1$} & \multicolumn{3}{c}{$\mathcal{P}_{\alpha}$}\\
        \cmidrule(lr){4-6}\cmidrule(lr){7-9}
                   &              & &Its   & CPU &Err   & Its  & CPU &Err   \\
        \midrule
		{\tt uniform} &48  & 55,296      & 29 & 0.82   & 1.348e-2 & 4  & 0.09 & 1.348e-2\\
		           &96  & 442,368     & 25 & 4.76   & 3.213e-3 & 4    & 0.51&3.213e-3 \\
 		           &192 & 3,538,944   & 26 & 51.03  & 8.003e-4 & 4    & 5.52& 8.003e-4 \\
		           &384 & 28,311,552  & 27 & 664.46 & 1.997e-4 & 4 
  & 76.34&1.997e-4 \\
		\hline
		{\tt nonuniform}  &36  &23,328     &21  &0.29   &7.679e-3  & 4 &0.05& 7.679e-3\\
		               &72  &186,624    &24  &1.46   & 1.476e-3  &4&0.18& 1.476e-3     \\
		                 &144 &1,492,992  &27  &14.05  & 4.339e-4   &4&1.47 & 4.339e-4 \\
		               &288 &11,943,936 &29  &154.03 &1.259e-4  &4 &16.01 &1.259e-4\\
		\hline
	\end{tabular}
\label{tab1}
\end{table}

\begin{table}[ht]\tabcolsep = 5.0pt
\caption{Convergence of different right-preconditioned GMRES(40) methods with $N_t = N_s = 2N_v$ and $\alpha = 10^{-3}$ (Set \uppercase\expandafter{\romannumeral2})}
\centering
\begin{tabular}{ccccccccc}
    \toprule
            Mesh   &$N_t$ &DoFs        &\multicolumn{3}{c}{$\mathcal{P}_1$} & \multicolumn{3}{c}{$\mathcal{P}_{\alpha}$}\\
        \cmidrule(lr){4-6}\cmidrule(lr){7-9}
                   &              & &Iters   & CPU &Err   & Its  & CPU &Err   \\
        \midrule
        {\tt uniform} & 50  &62,500    & 27 & 0.83 & 1.281e-2 & 4 & 
                      0.10 & 1.281e-2\\
                      & 100 &500,000    & 27 & 6.94 & 1.725e-3 & 4 & 0.64 & 1.725e-3 \\
 		         & 200 &4,000,000  & 28 & 62.60 & 6.168e-4 & 4 & 
                      6.28 &6.169e-4   \\
                      & 400 &32,000,000 & 28 & 962.07 &1.550e-4  & 4 & 96.51 & 1.551e-4   \\
		\hline
		{\tt nonuniform} & 40 &32,000      &21 & 0.33 & 2.640e-3 & 4 & 0.06 & 2.640e-3  \\		           
		              & 80  & 256,000    &24 & 2.26 & 7.929e-4 & 4 & 0.29 & 7.929e-4  \\
		              & 160 & 2,048,000  &24 & 16.98 & 1.862e-4 & 4 & 2.16 & 1.862e-4 \\
                        & 320 & 16,384,000 &25 &176.43 & 4.665e-5 & 4 & 22.05 & 4.665e-5 \\
        \hline
\end{tabular}
\label{tab2}
\end{table}
The performance of the right-preconditioned GMRES(40) method with the preconditioners $\mathcal{P}_1$ and $\mathcal{P}_{\alpha}$ is reported in Tables~\ref{tab1}--\ref{tab2} for the AaO system arising from the Crank--Nicolson temporal discretization using both uniform and nonuniform spatial meshes. The results correspond to Parameter Sets \uppercase\expandafter{\romannumeral1} and \uppercase\expandafter{\romannumeral2}, respectively.

Several observations can be made. First, both preconditioners produce accurate numerical solutions, yielding essentially identical relative errors for all mesh sizes considered. Second, the proposed preconditioner $\mathcal{P}_{\alpha}$ substantially reduces both the number of GMRES($40$) iterations and the overall computational time compared with the classical preconditioner $\mathcal{P}_1$. In particular, the iteration count remains equal to four for all discretization levels, demonstrating an excellent mesh-independent convergence behaviour. Finally, although the nonuniform spatial discretization achieves higher accuracy with fewer spatial degrees of freedom, it has virtually no influence on the convergence behaviour of either preconditioner. This observation agrees with the theoretical analysis, since the proposed preconditioners are constructed solely from the temporal discretization.

\begin{figure}[ht]
	\setlength{\tabcolsep}{0.2pt}
	\centering
	\begin{tabular}{m{0.4cm}<{\centering} m{5.3cm}<{\centering} m{5.3cm}<{\centering} m{5.3cm}<{\centering}}
		& $\mathcal{M}$ & $\mathcal{P}_{1}^{-1}\mathcal{M}$ & $\mathcal{P}_{\alpha}^{-1}\mathcal{M}$ \\
		\rotatebox{90}{{\tt uniform}} &
		\includegraphics[width=2.0in,height=1.9in]{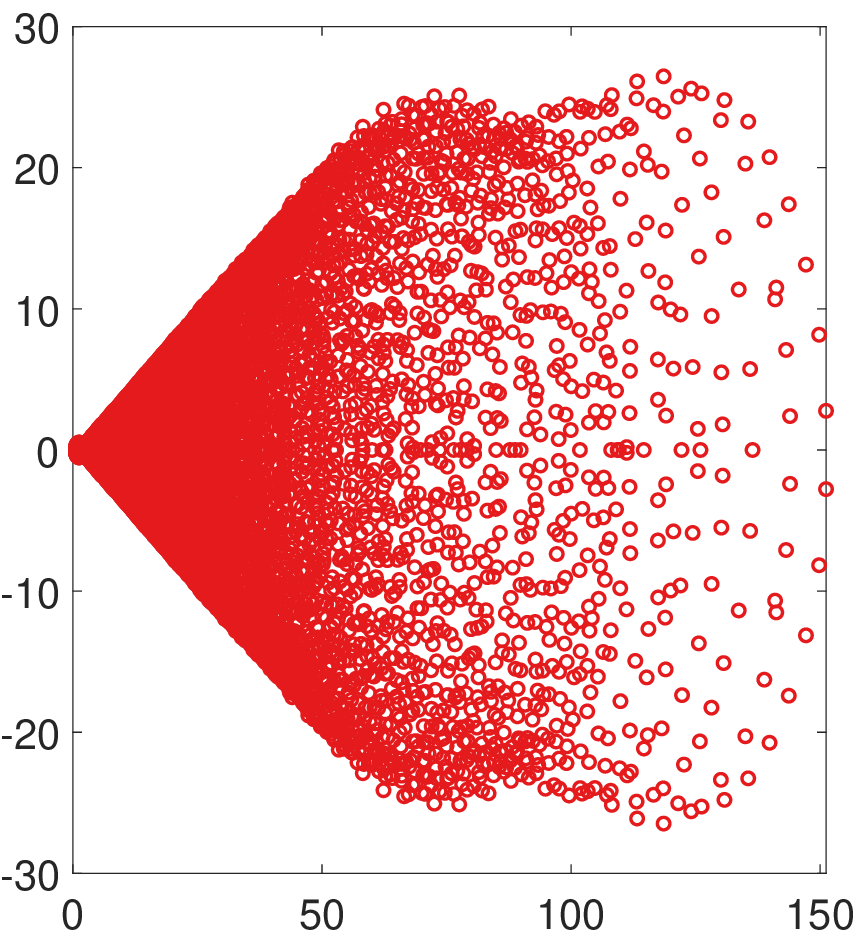} &
		\includegraphics[width=2.0in,height=1.9in]{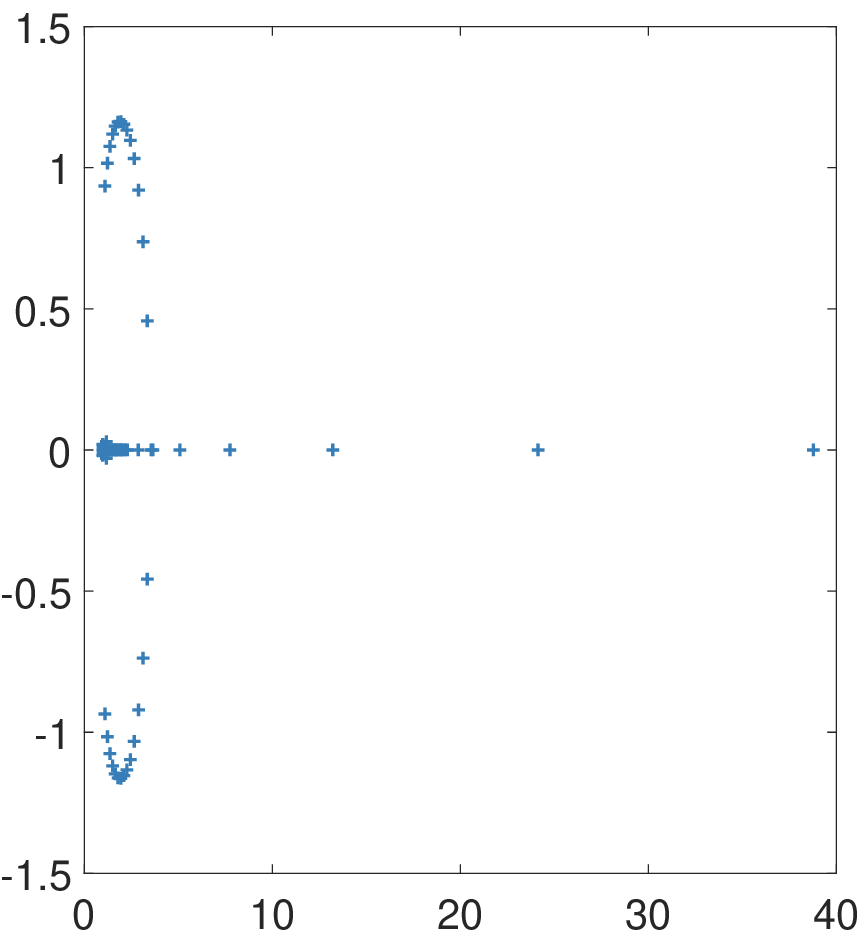} &
		\includegraphics[width=2.0in,height=1.9in]{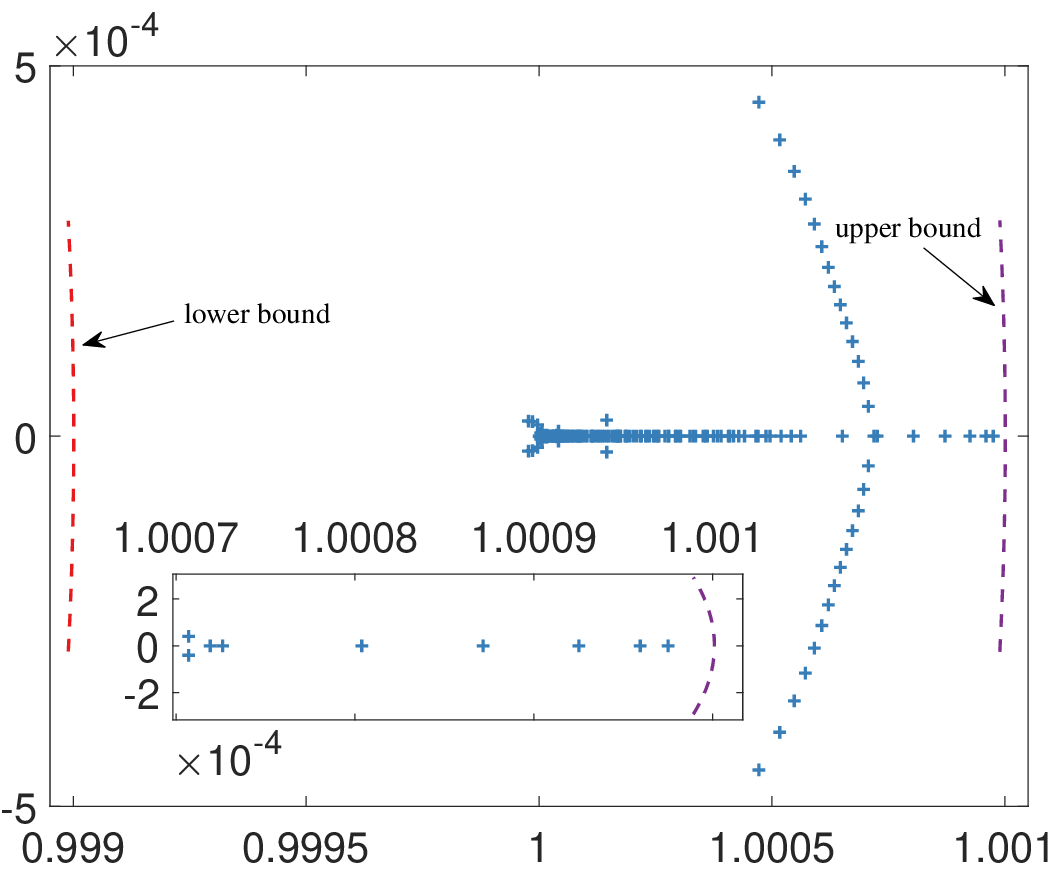} \\
		\rotatebox{90}{{\tt nonuniform}} &
		\includegraphics[width=2.0in,height=1.9in]{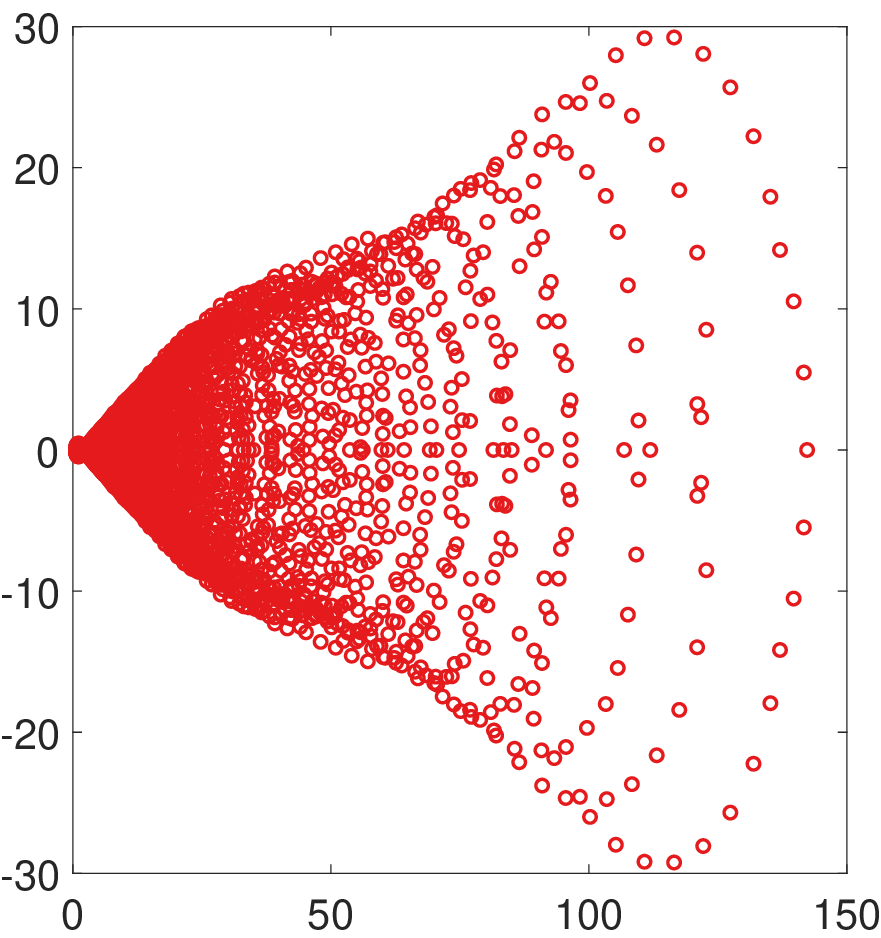} &
		\includegraphics[width=2.0in,height=1.9in]{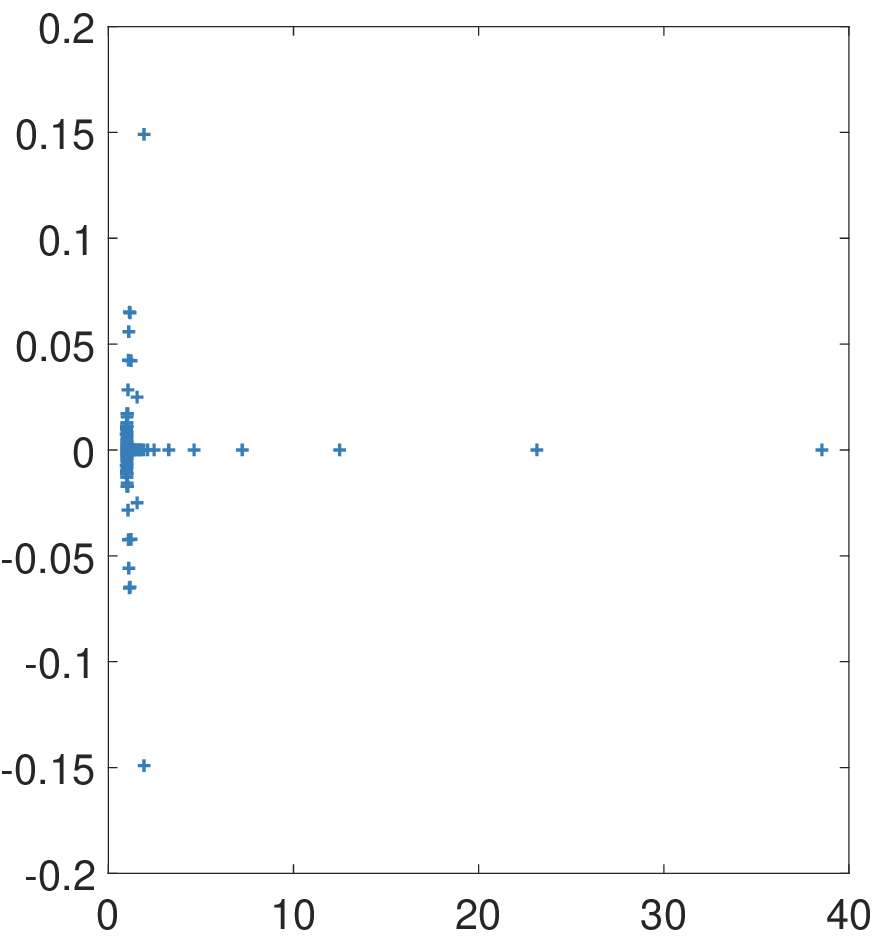} &
		\includegraphics[width=2.0in,height=1.9in]{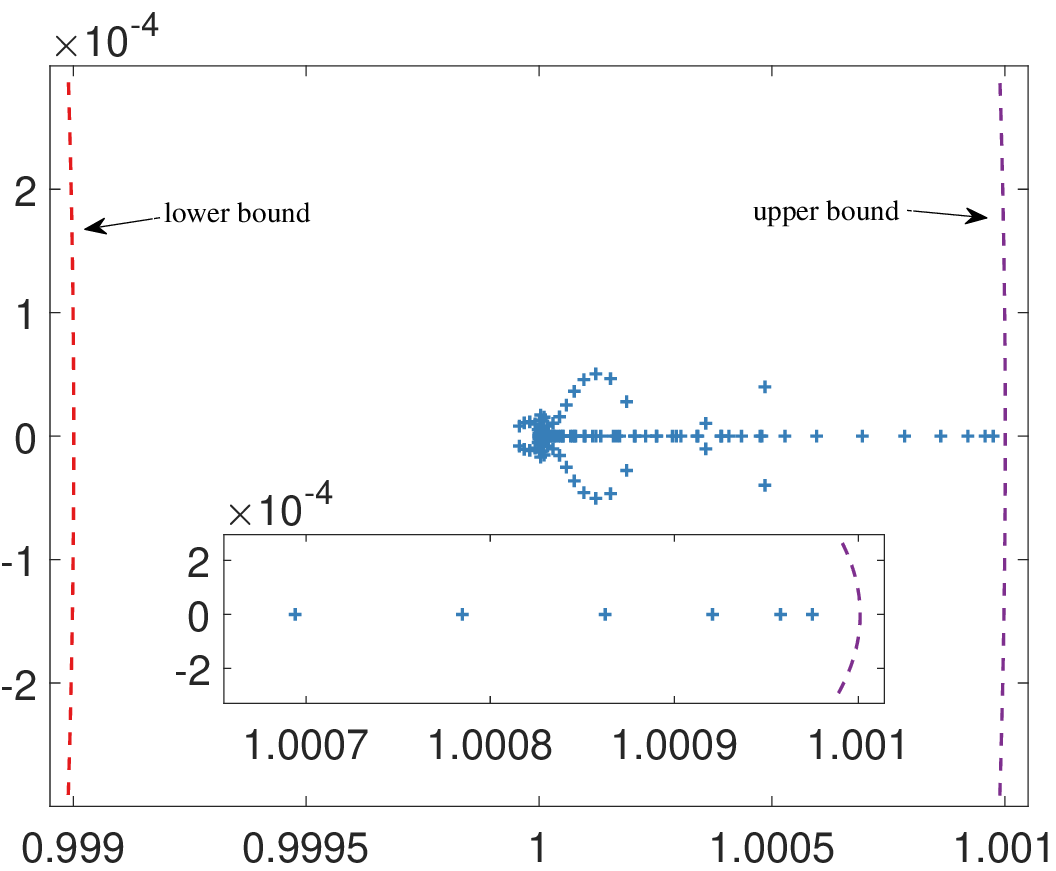} \\
	\end{tabular}
	\caption{Eigenvalue distributions of the original and preconditioned matrices, $\mathcal{M}$, $\mathcal{P}^{-1}_1\mathcal{M}$ and $\mathcal{P}^{-1}_{\alpha}\mathcal{M}$, for Parameter Set \uppercase\expandafter{\romannumeral1} using a uniform spatial discretization with $N_t=N_s=2N_v=36$ and $\alpha=10^{-3}$.}
	\label{fig1}
\end{figure}

\begin{figure}[ht]
\setlength{\tabcolsep}{0.2pt}
	\centering
	\begin{tabular}{m{0.4cm}<{\centering} m{5.3cm}<{\centering} m{5.3cm}<{\centering} m{5.3cm}<{\centering}}
		& $\mathcal{M}$ & $\mathcal{P}_{1}^{-1}\mathcal{M}$ & $\mathcal{P}_{\alpha}^{-1}\mathcal{M}$ \\
		\rotatebox{90}{{\tt uniform}} &
		\includegraphics[width=2.0in,height=1.9in]{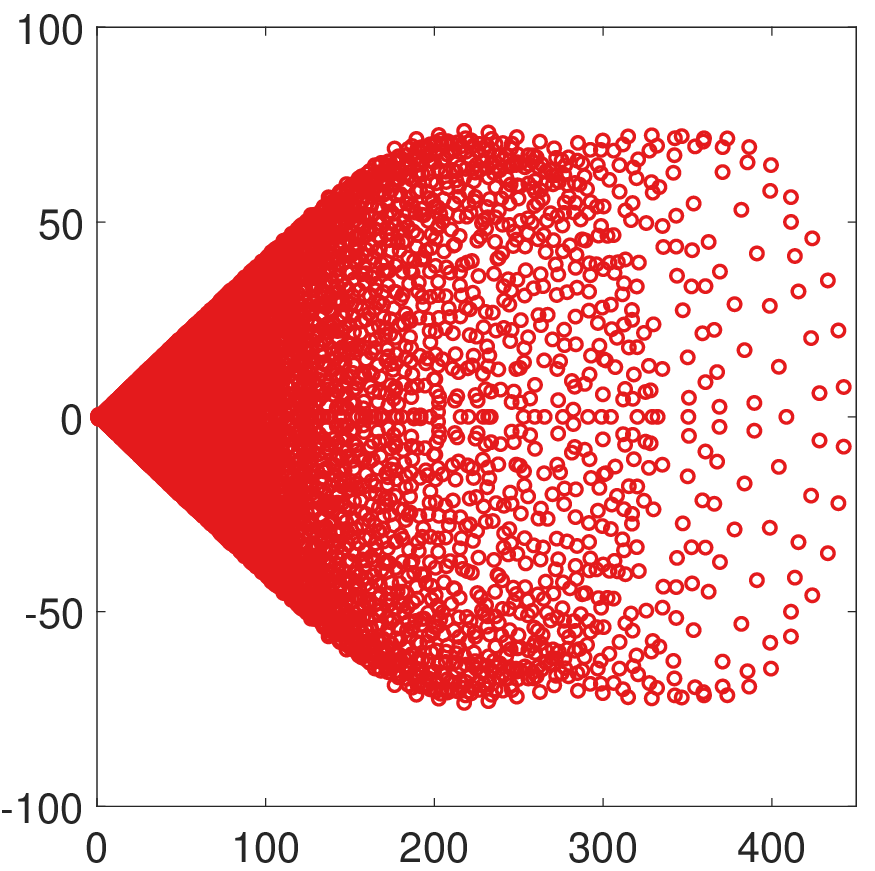} &
		\includegraphics[width=2.0in,height=1.9in]{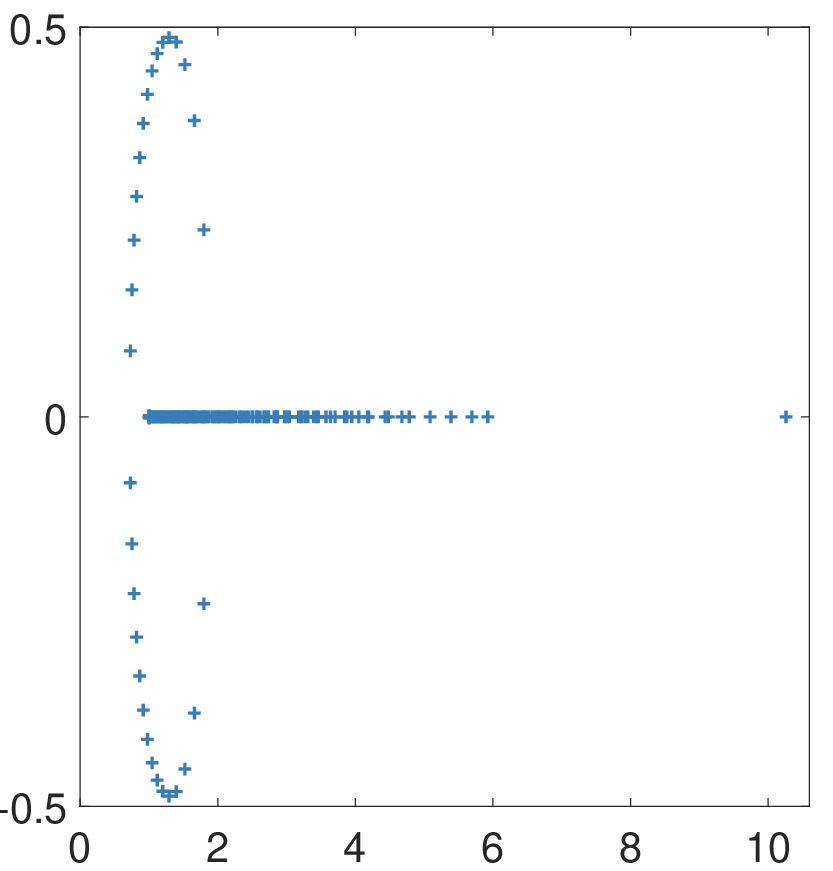} &
		\includegraphics[width=2.0in,height=1.9in]{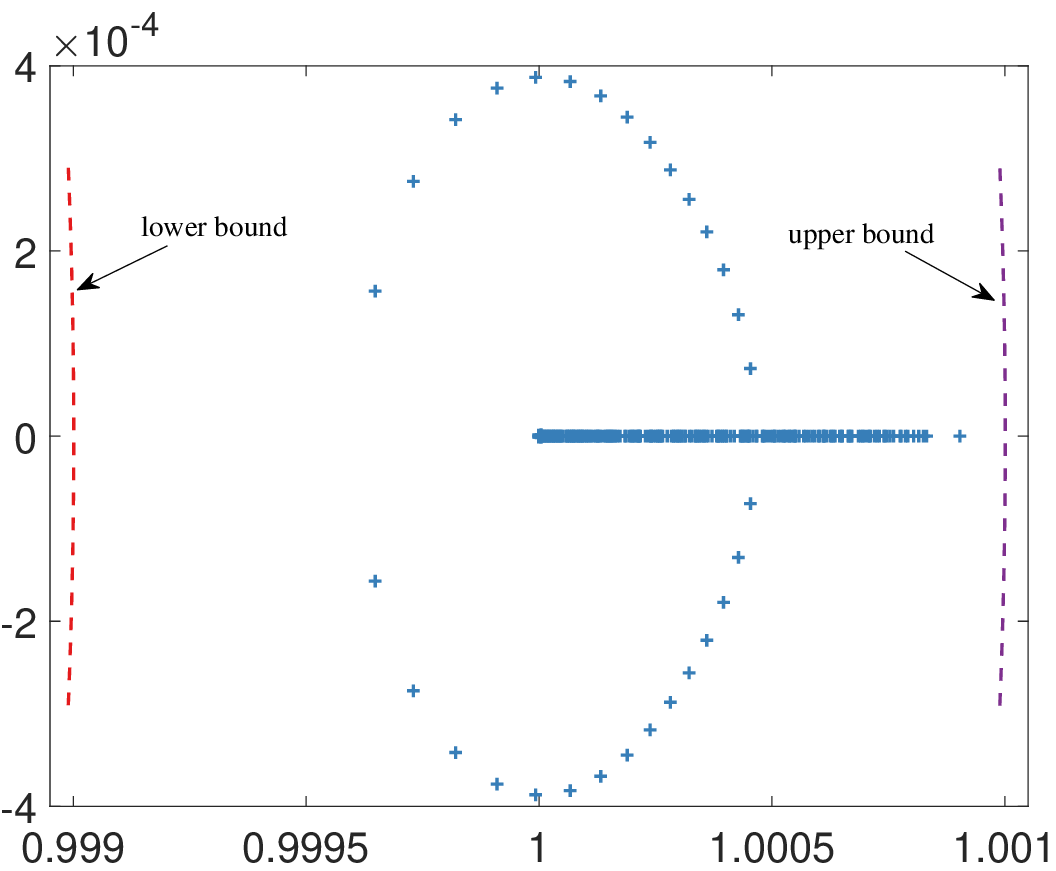} \\
		\rotatebox{90}{{\tt nonuniform}} &
		\includegraphics[width=2.0in,height=1.9in]{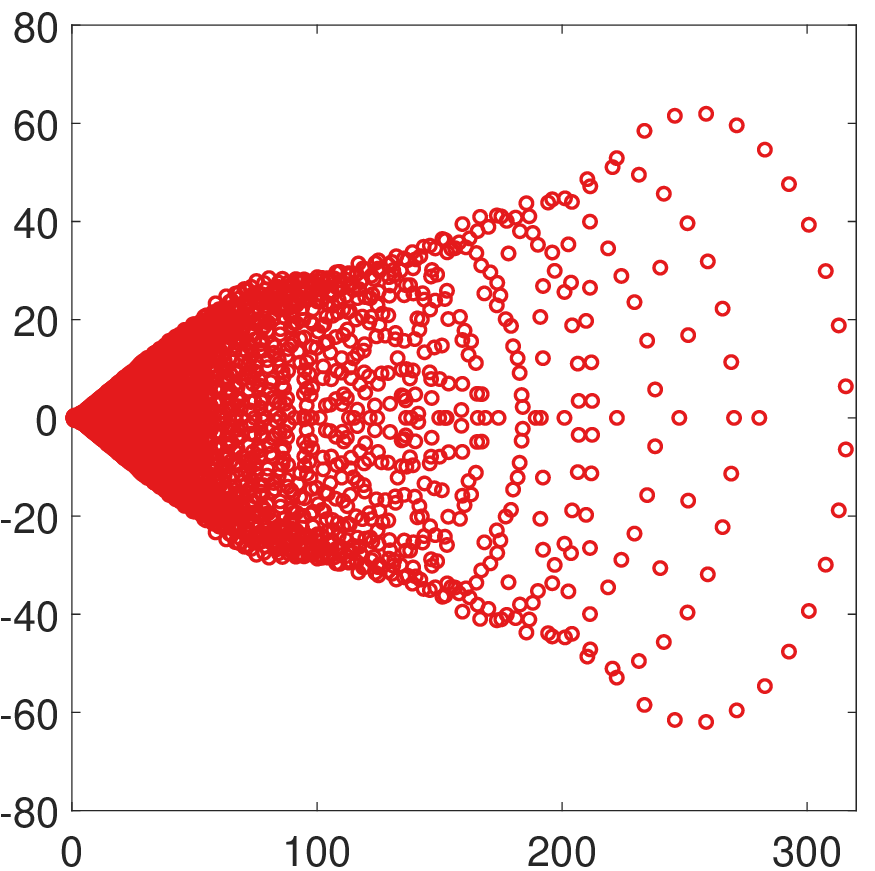} &
		\includegraphics[width=2.0in,height=1.9in]{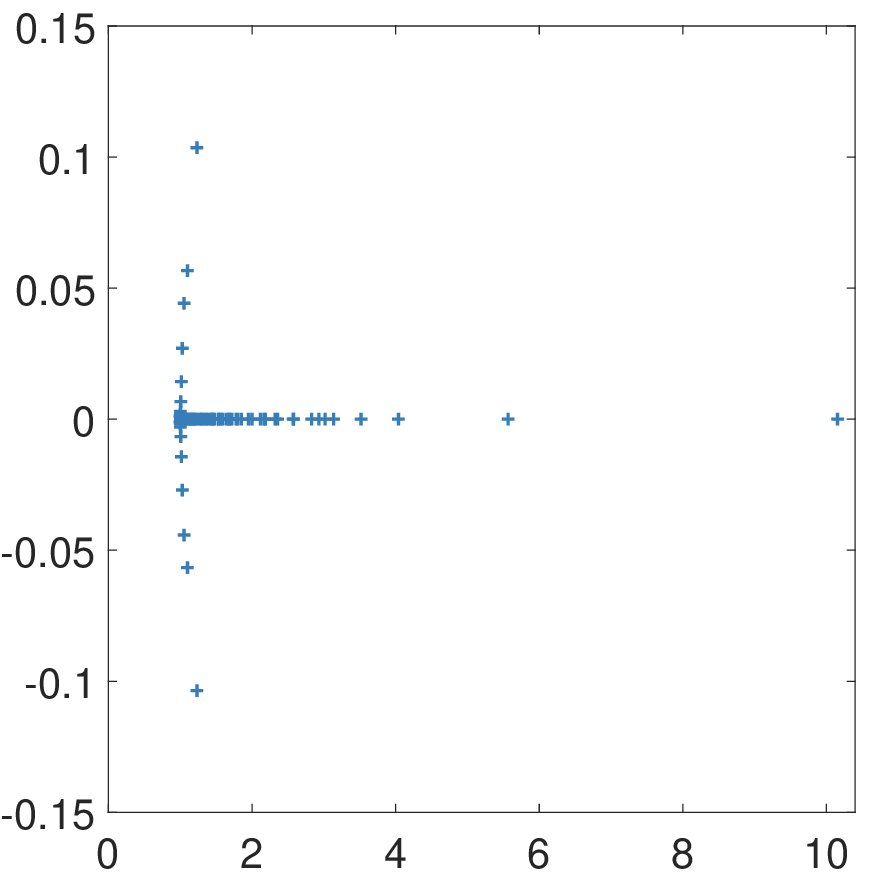} &
		\includegraphics[width=2.0in,height=1.9in]{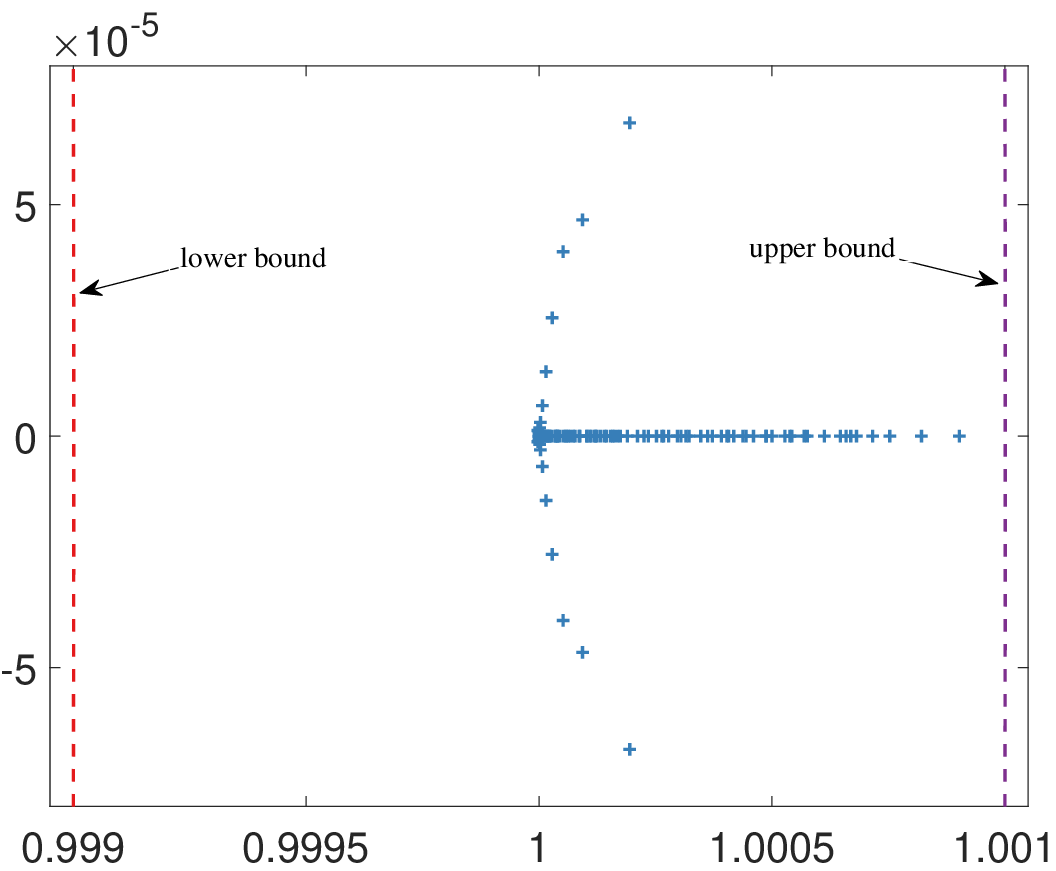} \\
	\end{tabular}
	\caption{Eigenvalue distributions of the original and preconditioned matrices, $\mathcal{M}$, $\mathcal{P}^{-1}_1 \mathcal{M}$ and $\mathcal{P}^{-1}_{\alpha}\mathcal{M}$, for Parameter Set \uppercase\expandafter{\romannumeral2} using a nonuniform spatial discretization with $N_t=N_s=2N_v=36$ and $\alpha=10^{-3}$.}
\label{fig2}
\end{figure}

Figures~\ref{fig1}--\ref{fig2} present the eigenvalue distributions of the original and preconditioned systems for the two Heston test problems, thereby illustrating the spectral results established in Theorem~\ref{th3.1}. As expected, the spectrum of $\mathcal{P}_{\alpha}^{-1}\mathcal{M}$ is significantly more tightly clustered around one than that of $\mathcal{P}_{1}^{-1}\mathcal{M}$, explaining the substantial reduction in the number of GMRES($40$) iterations observed in Tables~\ref{tab1}--\ref{tab2}. Moreover, all eigenvalues of $\mathcal{P}_{\alpha}^{-1}\mathcal{M}$ are confined to the annulus $\Omega_{\alpha}$, which is in complete agreement with the theoretical spectral bounds derived in Theorem~\ref{th3.1}.
\subsection{Stochastic Alpha Beta and Rho (SABR) model}

As a second test problem, we consider the PDE associated with the SABR model \cite{Graaf,Sydow19}, which is widely used in the modelling of interest-rate and foreign-exchange derivatives. The governing PDE is given by
\begin{equation}\label{eq5.7}
\frac{\partial u}{\partial t}
=
\frac{1}{2}v^2s^{2\beta}D^{2(1-\beta)}
\frac{\partial^2u}{\partial s^2}
+
\rho\sigma s^{\beta}D^{(1-\beta)}v^2
\frac{\partial^2u}{\partial s\partial v}
+
\frac{1}{2}\sigma^2v^2
\frac{\partial^2u}{\partial v^2}
+
rs\frac{\partial u}{\partial s}
-
ru,
\quad
D(t)=e^{-rt},
\end{equation}
for $s>0$, $v>0$ and $0\le t<T$.

The computational domain is chosen in the same manner as in Section \ref{sec5.1}. The initial condition and the first three boundary conditions are identical to those given in Eqs.~\eqref{eq5.2}-\eqref{eq5.5}. The remaining boundary condition, corresponding to the degenerate boundary at $v=0$, is given by
\begin{equation*}
\frac{\partial u(s,0,t)}{\partial t}
=
rs\frac{\partial u(s,0,t)}{\partial s}
-
ru(s,0,t).
\end{equation*}

After applying the spatial discretization, the resulting all-at-once linear system given in \eqref{eq2.2} is solved using the proposed preconditioned iterative method. The relative error is again defined by
\begin{equation*}
{\rm Err}
=
\frac{\left|{\rm Price}(S_0,V_0)-{\rm Price}_{\rm ref}\right|}
{\left|{\rm Price}_{\rm ref}\right|},
\end{equation*}
where the reference value ${\rm Price}_{\rm ref}$ is taken from Problem~1 in \cite[Section~4.1]{Sydow19} and is computed using the script {\tt Script\_Compare\_European.m} contained in the package\footnote{The source code is available at \url{https://github.com/jkirkby3/PROJ_Option_Pricing_Matlab}.}. If the point $(S_0,V_0)$ does not coincide with a grid point, the corresponding option value is obtained by bilinear interpolation; see \cite[p.~68]{Graaf} for details.

The numerical experiments are performed for the following two sets of SABR model parameters \cite{Sydow19}:
\begin{itemize}
\item Set \uppercase\expandafter{\romannumeral3}: $S_{max}=24$, $V_{max}=3$, $T=2$, $\beta=0.5$, $\sigma=0.4$, $r=0.0$, $\rho=0.0$, $K=S_0=0.5$, $V_0=0.5$;
\item Set \uppercase\expandafter{\romannumeral4}: $S_{max}=4$, $V_{max}=2$, $T=10$, $\beta=0.5$, $\sigma=0.8$, $r=0.0$, $\rho=-0.6$, $K=S_0=0.07$, $V_0=0.4$.
\end{itemize}

\begin{table}[ht]\tabcolsep=5.0pt
\caption{Convergence of different right-preconditioned GMRES(40) methods with $N_t=N_s=2N_v$ and $\alpha=10^{-3}$.}
\centering
\begin{tabular}{ccccccccc}
    \toprule
    Set   &$N_t$ &DoFs        &\multicolumn{3}{c}{$\mathcal{P}_1$} & \multicolumn{3}{c}{$\mathcal{P}_{\alpha}$}\\
    \cmidrule(lr){4-6}\cmidrule(lr){7-9}
                   &              & &Its   & CPU &Err   & Its  & CPU &Err   \\
    \midrule
     \uppercase\expandafter{\romannumeral3} & 48  &55,296   & $\ddag$ & --- & --- & 4 & 0.08 & 1.365e-1\\
                      & 96 &442,368    & $\ddag$ & --- & --- & 4 & 0.38 & 3.102e-2 \\
 		         & 192 &3,538,944  & $\ddag$ & --- & --- & 4 & 
                      3.11 &7.003e-3   \\
                      & 384 &28,311,552 & $\ddag$ & --- & --- & 4 & 33.45 & 1.550e-3   \\
    \hline
    \uppercase\expandafter{\romannumeral4} & 48 &55,296 &$\ddag$ & --- &--- & 4 &0.09 & 4.760e-2\\
    &96  &442,368    & $\ddag$ & --- & --- & 4 &0.40  &2.191e-2\\
    &192 &3,538,944  & $\ddag$ & --- & --- & 4 &3.68  &7.892e-3\\
    &384 &28,311,552 & $\ddag$ & --- & --- & 4 &37.61 &2.839e-3\\
\hline
\end{tabular}
\label{tab3}
\end{table}

Table~\ref{tab3} reports the performance of the right-preconditioned GMRES(40) method with the preconditioners $\mathcal{P}_1$ and $\mathcal{P}_{\alpha}$ for the SABR model using Parameter Sets \uppercase\expandafter{\romannumeral3} and \uppercase\expandafter{\romannumeral4}. The symbol ``$\ddag$'' indicates that the preconditioner $\mathcal{P}_1$ is singular, causing the corresponding GMRES(40) iteration to fail. Consequently, neither the CPU time nor the relative error can be reported; see Remark~\ref{remark1x} for further discussion.

In contrast, the proposed preconditioner $\mathcal{P}_{\alpha}$ remains fully effective. For both parameter sets, the GMRES(40) method converges in four iterations for all discretization levels considered, demonstrating a mesh-independent convergence behaviour. These results clearly illustrate the robustness of the proposed preconditioner with respect to both the temporal and spatial discretization parameters.

To further illustrate the theoretical results, Fig.~\ref{fig3} displays the eigenvalue distributions of the original and preconditioned systems. The eigenvalues of $\mathcal{P}_{\alpha}^{-1}\mathcal{M}$ are strongly clustered around one and satisfy the bounds established in Theorem~\ref{th3.1}. Figure~\ref{fig4} shows that the spatial discretization matrix $A$ possesses several zero eigenvalues. As discussed in Remark~\ref{remark1x}, this causes the classical preconditioner $\mathcal{P}_1$ to become singular, explaining its failure in Table~\ref{tab3}. Moreover, the presence of these zero eigenvalues shows that the upper spectral bound derived in Theorem~\ref{th3.1} is attained, confirming the sharpness of the theoretical analysis; see Remark~\ref{remark2x} and Fig.~\ref{fig3}.

Overall, these experiments demonstrate that the generalized block $\alpha$-circulant preconditioner is both robust and efficient for the SABR model. In particular, it successfully handles situations in which the classical block circulant preconditioner becomes singular while maintaining excellent convergence properties.

\begin{figure}[ht]
\setlength{\tabcolsep}{0.2pt}
	\centering
	\begin{tabular}{m{0.4cm}<{\centering} m{5.3cm}<{\centering} m{5.3cm}<{\centering}}
		& $\mathcal{M}$ & $\mathcal{P}^{-1}_{\alpha}\mathcal{M}$ \\
		\rotatebox{90}{Set \uppercase\expandafter{\romannumeral3}} &
		\includegraphics[width=2.0in,height=1.9in]{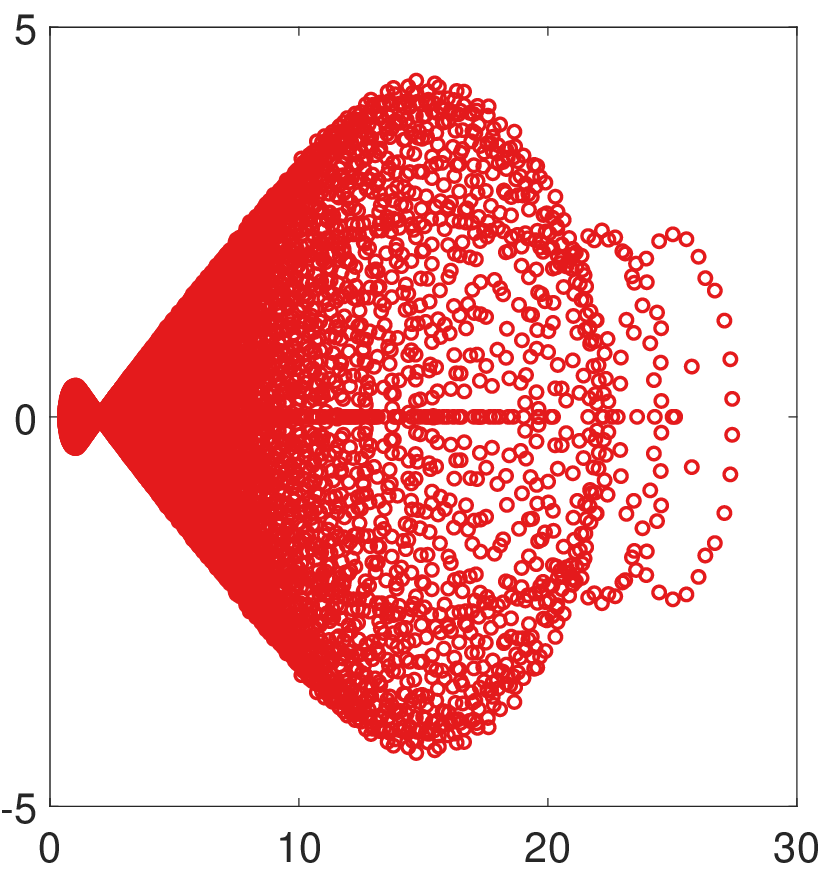} &
		\includegraphics[width=2.0in,height=1.9in]{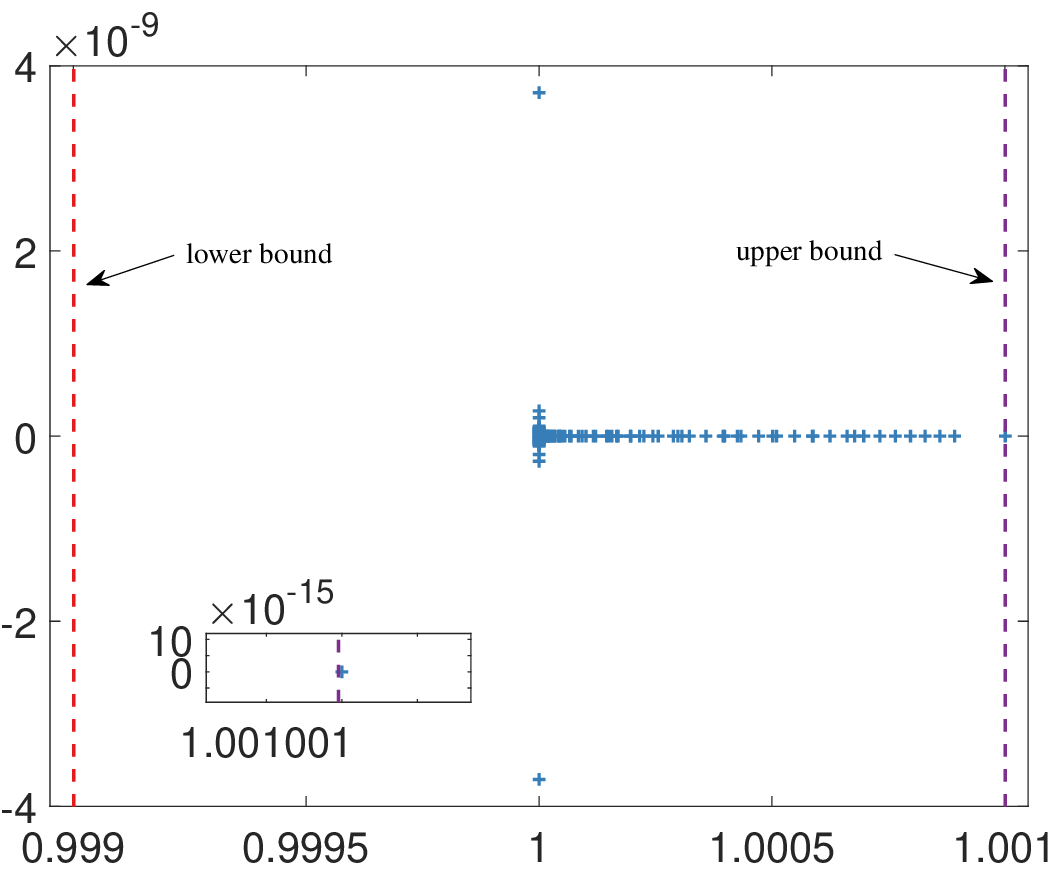} \\
		\rotatebox{90}{Set \uppercase\expandafter{\romannumeral4}} &
		\includegraphics[width=2.0in,height=1.9in]{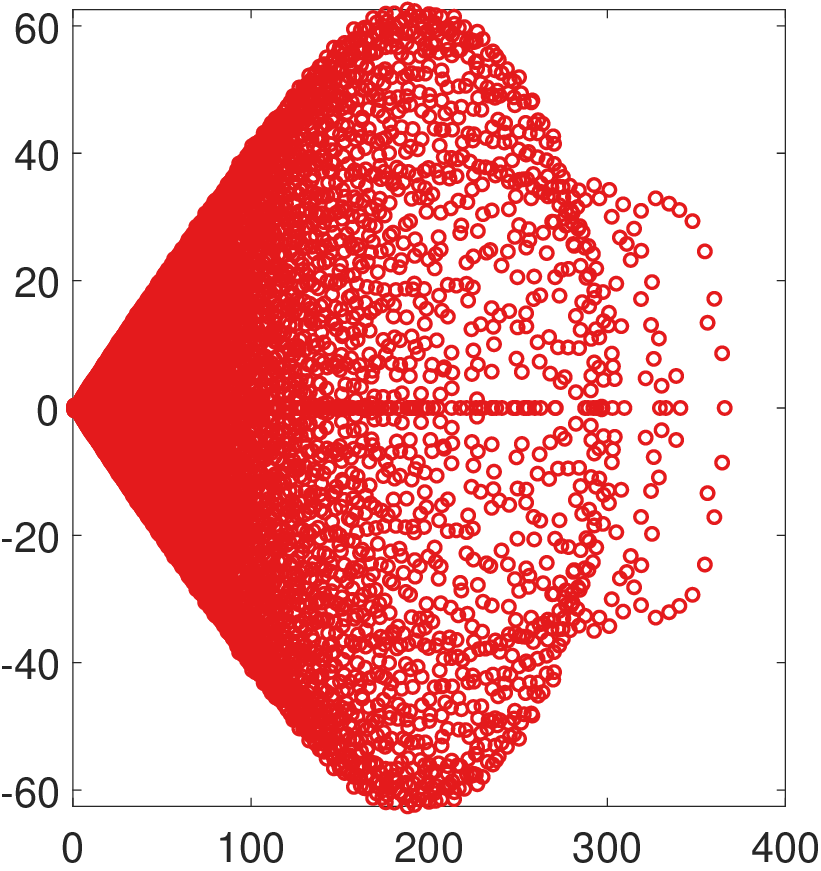} &
		\includegraphics[width=2.0in,height=1.9in]{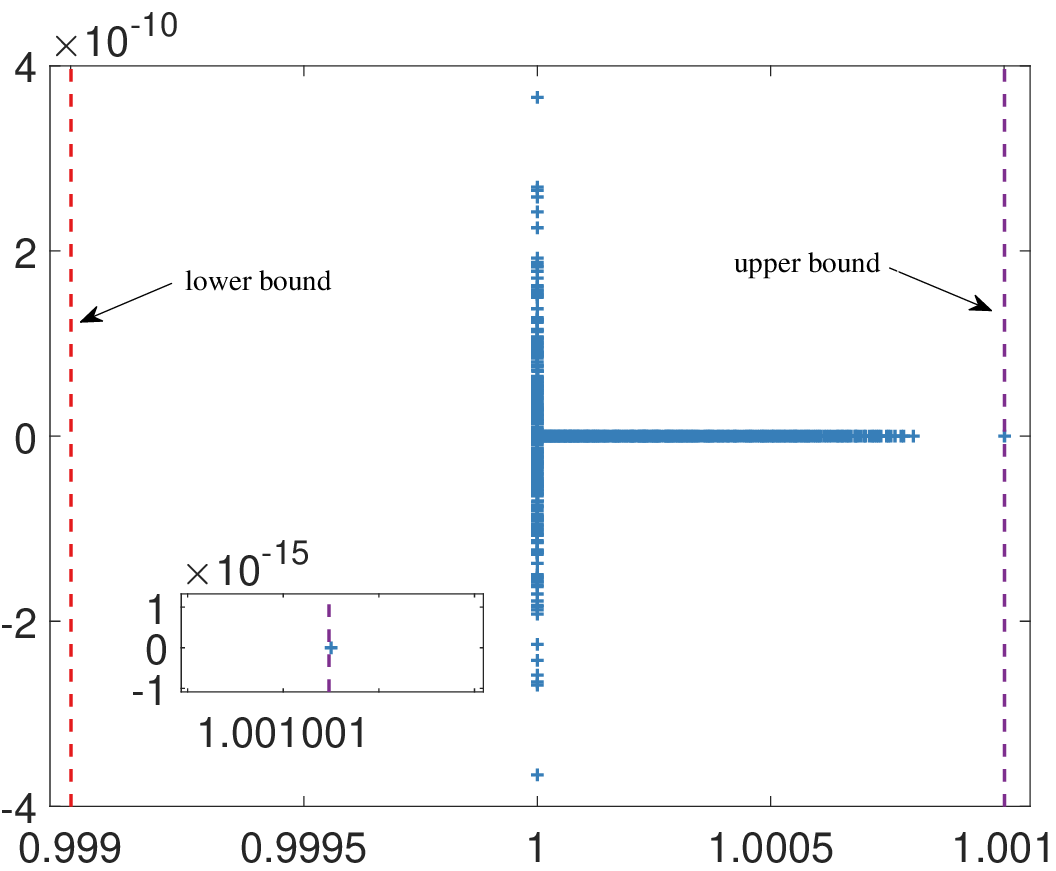} \\
	\end{tabular}
\caption{Eigenvalue distributions of the original and preconditioned matrices, $\mathcal{M}$ and $\mathcal{P}^{-1}_{\alpha}\mathcal{M}$, for Parameter Sets \uppercase\expandafter{\romannumeral3} and \uppercase\expandafter{\romannumeral4} with $N_t=N_s=2N_v=36$ and $\alpha=10^{-3}$.}
\label{fig3}
\end{figure}

\begin{figure}[ht]
\setlength{\tabcolsep}{0.2pt}
\centering
\begin{tabular}{m{0.4cm}<{\centering} m{11cm}<{\centering}}
\rotatebox{90}{Set \uppercase\expandafter{\romannumeral3}} &
\includegraphics[width=4.0in,height=0.8in]{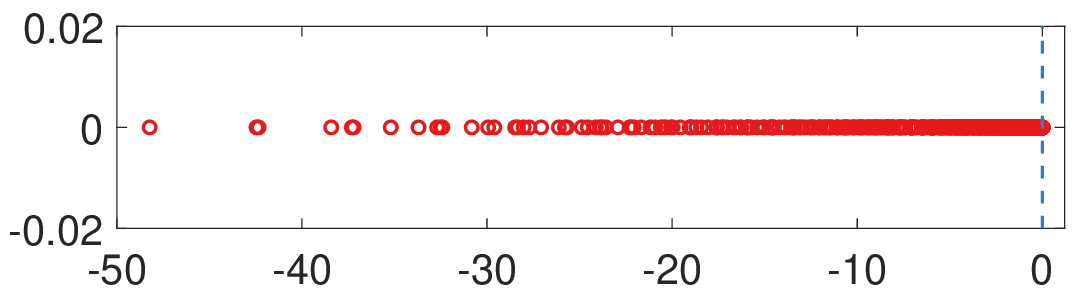} \\
\rotatebox{90}{Set \uppercase\expandafter{\romannumeral4}} &
\includegraphics[width=4.0in,height=0.8in]{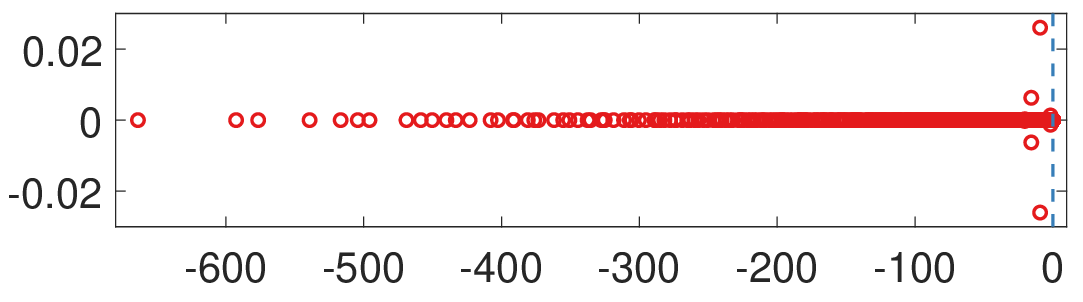}
\end{tabular}
\caption{Eigenvalues of the spatial discretization matrix $\tilde{A}$ for Parameter Sets \uppercase\expandafter{\romannumeral3} and \uppercase\expandafter{\romannumeral4} with $N_t=N_s=2N_v=36$.}
\label{fig4}
\end{figure}

We next consider the SABR model with a non-zero interest rate ($r\neq0$). In this case, the spatial operator in Eq.~\eqref{eq1.1} becomes time dependent. More precisely,
\[
\mathcal{L}(t)=d(t)\mathcal{L},
\]
where $d(t)>0$ is the known function appearing in Eq.~\eqref{eq5.7}. Following the derivation in Section~\ref{sec2}, the corresponding Crank--Nicolson AaO system is obtained by replacing $\mathcal{M}$ and $\bm{\xi}$ in Eq.~\eqref{eq2.2} by
\begin{equation}
\mathcal{M}
=
B_1\otimes I_s
-
(\tilde{D}B_2)\otimes\tilde{A},
\qquad
\bm{\xi}
=
\left[
\left[
\left(
I_s+\frac{d(t_{\frac12})}{2}\tilde{A}
\right)\bm{u}^0
\right]^{\!\top},
\bm{0}^{\top},
\ldots,
\bm{0}^{\top}
\right]^{\!\top},
\label{eq5.8}
\end{equation}
where
\[
\tilde{D}
=
\operatorname{diag}
\left(
d(t_{\frac12}),
d(t_{\frac32}),
\ldots,
d(t_{N_t-\frac12})
\right).
\]

Since $d(t)=e^{-rt}$ is a smooth function and both $r$ and $T$ are typically small in practical applications, we approximate $\tilde{D}$ by its average value,
\[
\bar d
=
\frac{1}{N_t}
\sum_{k=0}^{N_t-1}
d\!\left(t_{k+\frac12}\right),
\]
so that
\[
\tilde{D}B_2\approx\bar d\,B_2.
\]
Consequently, only a minor modification of the proposed PinT preconditioner \eqref{eq3.1} is required, namely replacing $C_2^{(\alpha)}$ by $\bar d\,C_2^{(\alpha)}$. The revised preconditioner therefore retains the same implementation strategy as described in Eq.~\eqref{eq3.2}.

The performance of the right-preconditioned GMRES(40) method with the preconditioners $\mathcal{P}_1$ and $\mathcal{P}_{\alpha}$ for the modified AaO system \eqref{eq5.8} is reported in Table~\ref{tab4}. The computations correspond to the following parameter set.

\begin{table}[ht]\tabcolsep=5.0pt
\caption{Convergence of different right-preconditioned GMRES(40) methods with $N_t = N_s = 2N_v$ and $\alpha = 10^{-3}$}
\centering
\begin{tabular}{ccccccccc}
  \toprule
  Set   &$N_t$ &DoFs        &\multicolumn{3}{c}{$\mathcal{P}_1$} & \multicolumn{3}{c}{$\mathcal{P}_{\alpha}$}\\
      \cmidrule(lr){4-6}\cmidrule(lr){7-9}
                   &              & &Its   & CPU &Err   & Its  & CPU &Err   \\
  \midrule
\uppercase\expandafter{\romannumeral5} & 48 & 55,296 & 100 &2.25 &2.282e-1  & 6 & 0.12 &2.282e-1\\
                                   & 96 & 442,368 & 111 & 19.86 &8.352e-2& 6 & 0.85 &8.352e-2 \\
                                   & 192 & 3,538,944 & 131 &179.11&9.245e-3&6&6.24&9.244e-3\\
                                   & 384 & 28,311,552 &\dag & ---&---&6&62.49&7.257e-4\\ 
\hline
\end{tabular}
\label{tab4}
\end{table}

\begin{itemize}
\item Set \uppercase\expandafter{\romannumeral5}: $S_{max}=23$, $V_{max}=3$, $T=2$, $\beta=0.45$, $\sigma=0.4$, $r=0.03$, $\rho=-0.65$, $K=S_0=0.5$, $V_0=0.2$.
\end{itemize}

Although the theoretical analysis developed in the previous sections does not directly apply to the case of time-dependent coefficients, the numerical results in Table~\ref{tab4} remain highly encouraging. Both preconditioned GMRES(40) methods converge for this problem, while the proposed preconditioner $\mathcal{P}_{\alpha}$ consistently requires only six iterations, independent of the mesh size. In contrast, the iteration count for the classical preconditioner $\mathcal{P}_1$ increases significantly as the mesh is refined and eventually fails to converge for the finest discretization.

Figure~\ref{fig5} provides further insight into this behaviour. Although no spectral bounds are available for the modified preconditioned system, the eigenvalues of $\mathcal{P}_{\alpha}^{-1}\mathcal{M}$ remain substantially more tightly clustered than those of $\mathcal{P}_1^{-1}\mathcal{M}$. This observation is fully consistent with the iteration counts reported in Table~\ref{tab4} and again demonstrates the superior robustness and efficiency of the proposed generalized block $\alpha$-circulant preconditioner.

\begin{figure}[ht]
\setlength{\tabcolsep}{0.2pt}
	\centering
	\begin{tabular}{m{5.3cm}<{\centering} m{5.3cm}<{\centering} m{5.3cm}<{\centering}}
		\includegraphics[width=2.0in,height=1.9in]{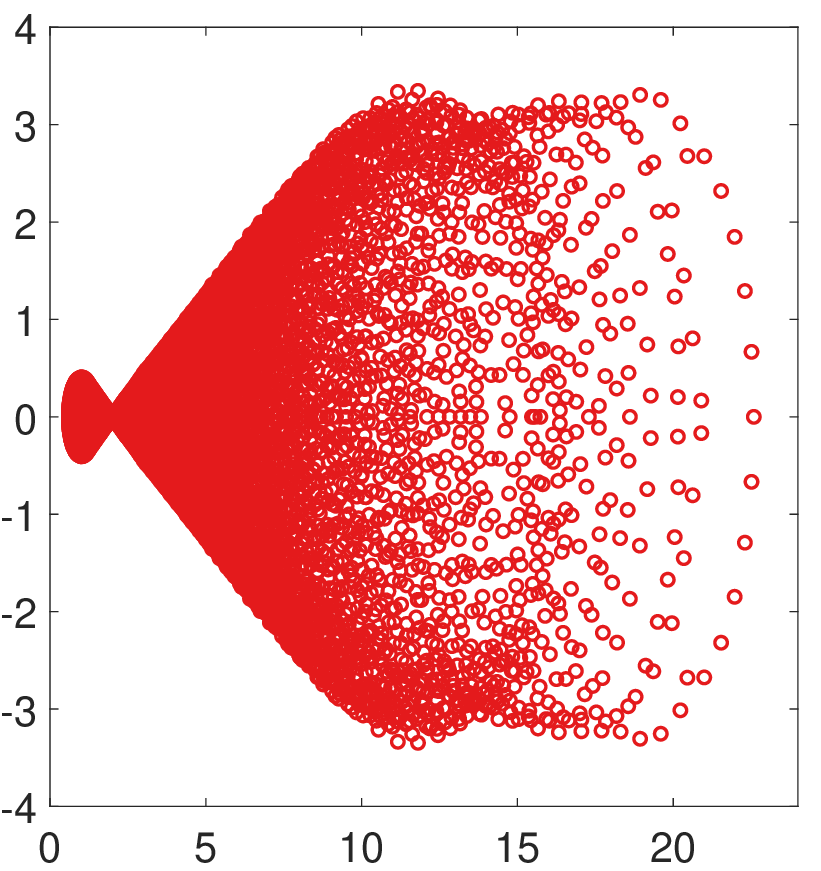} &
		\includegraphics[width=2.0in,height=1.9in]{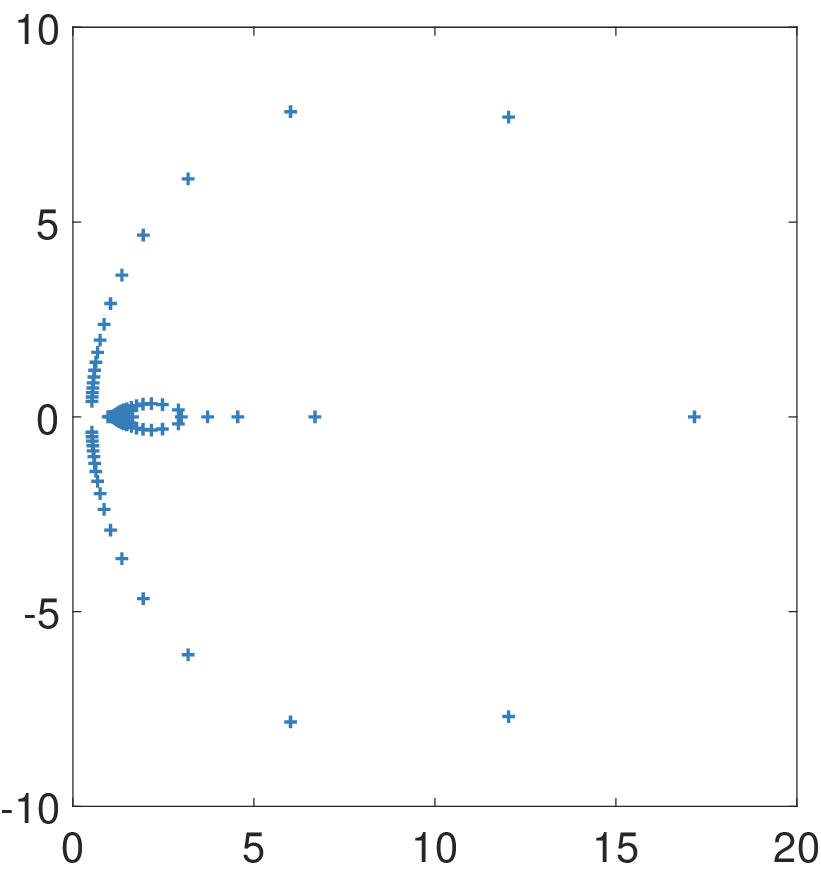} &
		\includegraphics[width=2.0in,height=1.9in]{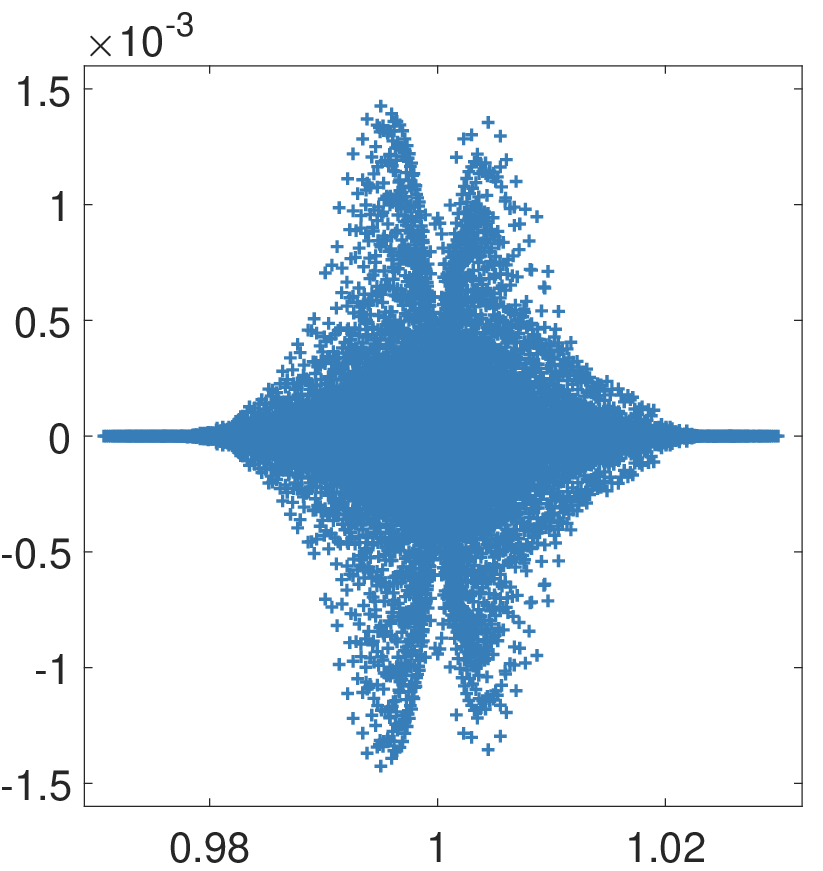} \\
		$\mathcal{M}$ & $\mathcal{P}_1^{-1}\mathcal{M}$ & $\mathcal{P}_{\alpha}^{-1}\mathcal{M}$
	\end{tabular}
\caption{Eigenvalue distributions of the original and preconditioned matrices, $\mathcal{M}$, $\mathcal{P}_1^{-1}\mathcal{M}$ and $\mathcal{P}_{\alpha}^{-1}\mathcal{M}$, for Parameter Set \uppercase\expandafter{\romannumeral5} with $N_t=N_s=2N_v=36$ and $\alpha=10^{-3}$.}
\label{fig5}
\end{figure}

\section{Concluding remarks}
\label{sec6}

This paper has presented a generalized block circulant (gBC) preconditioner for the efficient parallel-in-time solution of all-at-once systems arising from the Crank--Nicolson discretization of evolutionary PDEs. The proposed preconditioner admits an efficient parallel implementation and is applicable to a broad class of evolutionary problems, including the option pricing PDEs considered in this paper.

A detailed spectral analysis has been established for the preconditioned systems. In particular, we have shown that most eigenvalues are equal to one, while the remaining eigenvalues are confined to a uniformly bounded annulus determined solely by the parameter $\alpha$. These sharp eigenvalue bounds explain the excellent convergence behaviour observed in the numerical experiments and clarify why the classical block circulant preconditioner $\mathcal{P}_1$ may become inefficient or even singular for certain classes of problems. Furthermore, under a mild condition on $\alpha$, a mesh-independent convergence estimate for the preconditioned GMRES($m$) method has been derived.

The numerical experiments involving Heston and SABR option pricing models confirm the theoretical analysis. In all examples, the proposed gBC preconditioner significantly reduces the number of GMRES($m$) iterations and the computational time while maintaining the accuracy of the numerical solution. The experiments also demonstrate that the proposed approach remains effective for problems for which the classical block circulant preconditioner fails.

The present work also suggests several directions for future research. In particular, the favourable performance of the proposed preconditioner for evolutionary PDEs with time-dependent coefficients deserves a rigorous theoretical analysis. It would also be of interest to investigate the applicability of the proposed gBC preconditioning strategy to other classes of all-at-once systems and to related block Toeplitz problems arising in large-scale scientific computing.

\section*{Acknowledgments}
\addcontentsline{toc}{section}{Acknowledgments}
\label{sec7}
\textit{This work is supported by the National Natural Science Foundation of
China (No. 12401536) and Sichuan Science and Technology Program (No. 2024NSFSC0441).}
\section*{Data availability}
The data-sets generated and analyzed during the current study are not publicly available but are available from the authors on reasonable request.

\section*{Conflict of interests}
The authors declare that they have no known competing financial interests or personal relationships that could have appeared to influence the work reported in this paper.
\addcontentsline{toc}{section}{References}
\bibliography{Ref}

@article{Hout2010,
author = {In 't Hout, K. J. and Foulon, S.},
title = {{ADI} Finite Difference Schemes for Option Pricing in the {Heston} Model with Correlation},
journal = {Int. J. Numer. Anal. Mod.},
year = {2010},
volume = {7},
number = {2},
pages = {303-320},
}

@article{gu2021note,
	title={A Note on Parallel Preconditioning for the All-at-Once Solution of {R}iesz Fractional Diffusion Equations},
	author={Gu, Xian-Ming and Zhao, Yong-Liang and Zhao, Xi-Le and Carpentieri, Bruno and Huang, Yu-Yun},
	journal={Numer. Math. Theor. Meth. Appl.},
	volume={14},
	number={4},
	pages={893-919},
	year={2021},
	publisher={Global Science Press}
}

@book{Hairer96,
	title={{Solving Ordinary Differential Equations \uppercase\expandafter{\romannumeral2}: Stiff and Differential Algebraic Problems}},
	author={Hairer, Ernst and Wanner, Gerhard},
	year={1996},
	address={Berlin, Heidelberg},
	publisher={Springer-Verlag},
}

@article{McDonald18,
author = {McDonald, Eleanor and Pestana, Jennifer and Wathen, Andy},
title = {Preconditioning and Iterative Solution of All-at-Once Systems for Evolutionary Partial Differential Equations},
journal = {SIAM J. Sci. Comput.},
volume = {40},
number = {2},
pages = {A1012-A1033},
year = {2018},
}

@article{Eisenstat83,
author = {Eisenstat, Stanley C. and Elman, Howard C. and Schultz, Martin H.},
title = {Variational Iterative Methods for Nonsymmetric Systems of Linear Equations},
journal = {SIAM J. Numer. Anal.},
volume = {20},
number = {2},
pages = {345-357},
year = {1983},
doi = {10.1137/0720023},
}

@article{Sydow19,
author = {von Sydow, L. and Milovanovi\'{c}, S. and Larsson, E. and In't Hout, K. and Wiktorsson, M. and Oosterlee, C. W. and Shcherbakov, V. and Wyns, M. and Leitao, A. and Jain, S. and  Haentjens T. and Wald\'{e}n, J.},
title = {{BENCHOP - SLV}: the {BENCHmarking} project in Option Pricing – Stochastic and Local Volatility problems},
journal = {Int. J. Comput. Math.},
volume = {96},
number = {10},
pages = {1910-1923},
year = {2019},
publisher = {Taylor & Francis},
doi = {10.1080/00207160.2018.1544368},
}

@book{Duffy22,
      title = {Numerical Methods in Computational Finance: A Partial Differential Equation (PDE/FDM) Approach},
      publisher = {John Wiley \& Sons},
      address = {Chichester, UK},
      year = {2022},
      author = {Daniel J. Duffy},
}

@article{Liu2020,
author = {Liu, J. and Wu, S.-L.},
title = {A Fast Block $\alpha$-Circulant Preconditoner for All-at-Once Systems From Wave Equations},
journal = {SIAM J. Matrix Anal. Appl.},
volume = {41},
number = {4},
pages = {1912-1943},
year = {2020},
}

@book{Ascher08,
	title={Numerical Methods for Evolutionary Differential Equations},
	author={Uri M. Ascher},
	address={Philadelphia, PA},
	year={2008},
	publisher={SIAM}
}

@article{Falgout14,
author = {Falgout, R. D. and Friedhoff, S. and Kolev, Tz. V. and MacLachlan, S. P. and Schroder, J. B.},
title = {Parallel Time Integration with Multigrid},
journal = {SIAM J. Sci. Comput.},
volume = {36},
number = {6},
pages = {C635-C661},
year = {2014},
doi = {10.1137/130944230},
}

@article{Sheen03,
    author = {Sheen, Dongwoo and Sloan, Ian H. and Thom\'{e}e, Vidar},
    title = {A parallel method for time discretization of parabolic equations based on {Laplace} transformation and quadrature},
    journal = {IMA J. Numer. Anal.},
    volume = {23},
    number = {2},
    pages = {269-299},
    year = {2003},
    doi = {10.1093/imanum/23.2.269},
}

@book{Alfio2007,
  title={Topics in Matrix Analysis},
  author={Horn, R.A. and Johnson, C.R.},
  year={1991},
  publisher={Cambridge University Press},
  address = {New York, NY},
}

@article{Tadmor12,
	title={A review of numerical methods for nonlinear partial differential equations},
	author={E. Tadmor},
	journal={Bull. Amer. Math. Soc.},
        volume = {49},
	pages={507--554},
	year={2012},
	number = {4}
}

@article{Hon2023,
	title={A sine transform based preconditioned {MINRES} method for all-at-once systems from constant and variable-coefficient evolutionary {PDEs}},
	author={Hon, Sean and Fung, Po Yin and Dong, Jiamei and Serra-Capizzano, Stefano},
	journal={Numer. Algor.},
	volume={95},
	pages={1769--1799},
	year={2024},
	number = {4}
}

@misc{gander21paradiag,
      title={{ParaDiag}: parallel-in-time algorithms based on the diagonalization technique},
      author={Martin J. Gander and Jun Liu and Shu-Lin Wu and Xiaoqiang Yue and Tao Zhou},
      year={2021},
      eprint={2005.09158v4},
      archivePrefix={arXiv},
      primaryClass={math.NA},
      note = {arXiv preprint},
}

@misc{Gan2022,
  author = {Gan, Di and Zhang, Guo-Feng and Liang, Zhao-Zheng},
  title = {Parallel Preconditioning of All-at-Once Systems for Fractional Diffusion Equations with {Riesz} Fractional Derivatives},
  publisher = {SSRN},
  year = {2022},
  pages = {11},
  note = {{SSRN} preprint, 11 pages},
  doi = {10.2139/ssrn.4234422},
}

@article{Hon2022,
	title={Optimal block circulant preconditioners for block {Toeplitz} systems with application to evolutionary {PDEs}},
	author={Hon, Sean},
	journal={J. Comput. Appl. Math.},
	volume={407},
	pages={113965},
	year={2022},
}

@article{embree2023,
      title={Extending {Elman's} Bound for {GMRES}}, 
      author={Mark Embree},
      journal={Linear Algebra Appl.},
      volume={726},
      pages={54-70},
      year={2025},
}

@article{Lischke20,
	title={What is the fractional {Laplacian? A} comparative review with new results},
	author={A. Lischke and G. Pang and M. Gulian and F. Song and C. Glusa and X. Zheng and Z. Mao and W. Cai and M.M. Meerschaert and M. Ainsworth and G.E. Karniadakis},
	journal={J. Comput. Phys.},
	volume={404},
	pages={109009},
	year={2020},
	doi={10.1016/j.jcp.2019.109009}
}

@book{Evans2010,
	title={Partial Differential Equations},
	author={Lawrence C. Evans},
	edition={2nd},
	volume={19},
	address={Providence, RI},
        series={Graduate Series in Mathematics},
	year={2010},
	publisher={American Mathematical Society}
}

@article{Lions2001,
title = {R\'esolution d'{EDP} par un sch\'ema en temps $\ll$ parar\'eel $\gg$},
journal = {C. R. Acad. Sci. Paris},
volume = {332},
number = {7},
pages = {661-668},
year = {2001},
doi = {10.1016/S0764-4442(00)01793-6},
author = {Jacques-Louis Lions and Yvon Maday and Gabriel Turinici},
}

@phdthesis{Dolgov14phd,
  author  = {Sergey Dolgov},
  title   = {Tensor product methods in numerical simulation of high-dimensional dynamical problems},
  school  = {Universit\"{a}t Leipzig},
  year    = {2014},
  type    = {{PhD} dissertation},
  address = {Leipzig, Gemeny},
  note    = {154 pages},
}

@article{Wu2022,
author = {Wu, Shu-Lin and Zhou, Tao and Zhou, Zhi},
title = {A Uniform Spectral Analysis for a Preconditioned All-at-Once System from First-Order and Second-Order Evolutionary Problems},
journal = {SIAM J. Matrix Anal. Appl.},
volume = {43},
number = {3},
pages = {1331-1353},
year = {2022},
doi = {10.1137/21M145358X},
}

@article{Crank47,
	title={A practical method for numerical evaluation of solutions of partial differential equations of the heat conduction type},
	author={Crank, John and Nicolson, Phyllis},
	journal={Proc. Camb. Phil. Soc.},
	volume={43},
    number={1},
	pages={50--67},
	year={1947},
	publisher={Elsevier}
}

@article{WU2021,
title = {Parallel implementation for the two-stage {SDIRK} methods via diagonalization},
journal = {J. Comput. Phys.},
volume = {428},
pages = {110076},
year = {2021},
doi = {10.1016/j.jcp.2020.110076},
author = {Shu-Lin Wu and Tao Zhou},
}

@mastersthesis{Graaf,
  author  = {de Graaf, C.~S.~L.},
  title   = {{Finite Difference Methods in Derivatives Pricing under Stochastic Volatility Models}},
  school  = {Mathematisch Instituut, Universiteit Leiden},
  year    = {2012},
  type    = {Master thesis},
  address = {Leiden, Netherlands},
  pages = {69},
}

@book{golub2013matrix,
  title={Matrix Computations},
  author={Golub, Gene H and Van Loan, Charles F}, 
  edition = {4},
  year={2013},
  publisher={Johns Hopkins University Press},
  address={Baltimore, MD},
}

@phdthesis{Elman1982phd,
  author  = {Elman, Howard C},
  title   = {Iterative Methods for Large, Sparse, Nonsymmetric Systems of Linear Equations},
  school  = {Yale University},
  year    = {1982},
  type    = {{PhD} thesis},
  address = {New Haven},
}

@article{Goddard19,
	title={A Note On Parallel Preconditioning For All-At-Once
Evolutionary {PDEs}},
	author={Anthony Goddard and Andy Wathen},
	journal={Electron. Trans. Numer. Anal.},
	volume={51},
	pages={135--150},
	year={2019},
	publisher={Kent State University},
        doi = {10.1553/etna_vol51s135}
}

@book{saad2003iterative,
	title={Iterative Methods for Sparse Linear Systems},
	author={Saad, Yousef},
        edition={2nd},
	year={2003},
	address={Philadelphia, PA},
	publisher={SIAM}
}

@article{WATHEN2022,
title = {Some observations on preconditioning for non-self-adjoint and time-dependent problems},
journal = {Comput. Math. Appl.},
volume = {116},
pages = {176-180},
year = {2022},
doi = {10.1016/j.camwa.2021.05.037},
author = {Andy Wathen},
}

@article{Maday2008,
title = {Parallelization in time through tensor-product space–time solvers},
journal = {C. R. Math.},
volume = {346},
number = {1},
pages = {113-118},
year = {2008},
doi = {10.1016/j.crma.2007.09.012},
author = {Yvon Maday and Einar M. Rønquist},
}

@misc{embree2022,
      title={How Descriptive are {GMRES} Convergence Bounds?}, 
      author={Mark Embree},
      year={2022},
      eprint={2209.01231},
      archivePrefix={arXiv},
      primaryClass={math.NA},
      note = {arXiv preprint},
}

@article{lin2021all,
author = {Lin, X.-L. and Ng, M.},
title = {An All-at-Once Preconditioner for Evolutionary Partial Differential Equations},
journal = {SIAM J. Sci. Comput.},
volume = {43},
number = {4},
pages = {A2766-A2784},
year = {2021},
doi = {10.1137/20M1316354},
}

@article{titley2014,
  title={{GMRES} convergence bounds that depend on the right-hand-side vector},
  author={Titley-Peloquin, David and Pestana, Jennifer and Wathen, Andrew J},
  journal={IMA J. Numer. Anal.},
  volume={34},
  number={2},
  pages={462--479},
  year={2014},
  publisher={Oxford University Press}
}

@techreport{Chen:2008ifem,
	author = {Chen, L.},
	type = {Technical Report},
	title = {{$i$FEM}: an integrated finite element methods package in {MATLAB}},
    institution = {Department of Mathematics, University of California at Irvine},
    address = {Irvine, CA},
	note = {\url{https://github.com/lyc102/ifem}},
    year = {2009},
}

@article{Fang09,
author = {Fang, F. and Oosterlee, C. W.},
title = {A Novel Pricing Method for {European} Options Based on {Fourier}-Cosine Series Expansions},
journal = {SIAM J. Sci. Comput.},
volume = {31},
number = {2},
pages = {826-848},
year = {2009},
doi = {10.1137/080718061},
}

@article {Gardiner1992,
author = {Gardiner, J. D. and Wette, M. R. and Laub, A. J. and Amato, J. J. and Moler, C. B.},
title = {Algorithm 705: a {FORTRAN}-77 software package for solving the {S}ylvester matrix equation {$AXB^{\top}+CXD^{\top}=E$}},
journal = {ACM Trans. Math. Software},
fjournal = {Association for Computing Machinery. Transactions on Mathematical Software},
volume = {18},
year = {1992},
number = {2},
pages = {232--238},
}
\end{document}